\documentclass[11pt,a4paper,reqno]{amsart}

\usepackage[T1]{fontenc}
\usepackage{microtype, a4wide, textcomp,fbb}
\usepackage{
    amsmath,  amssymb,  amsthm,   amscd,
    gensymb,  graphicx, etoolbox, 
    booktabs, stackrel, mathtools    
}
\usepackage[usenames,dvipsnames]{xcolor}
\definecolor{darkblue}{rgb}{0,0,0.7}
\definecolor{darkred}{rgb}{0.7,0,0}
\usepackage[colorlinks=true, linkcolor=darkred, citecolor=darkblue, urlcolor=blue, pagebackref=true, breaklinks=true]{hyperref}
\usepackage[capitalise]{cleveref}
\usepackage{placeins}
\usepackage{relsize}
\usepackage{todonotes}
\usepackage{soul}
\usepackage{enumitem}%For adding label to enumeration

\usepackage{tikz}
\usetikzlibrary{arrows}
\usetikzlibrary{shapes}
\usepackage{float}
\usepackage{tabularx}
\usepackage{kantlipsum}
\usepackage{array}
\usepackage[margin=2.6cm]{geometry}

\newtheorem{proposition}{Proposition}[section]
\newtheorem{lemma}[proposition]{Lemma}
\newtheorem{theorem}[proposition]{Theorem}

\newtheorem{corollary}[proposition]{Corollary}

\newtheorem{conjecture}[proposition]{Conjecture}
\theoremstyle{definition}
\newtheorem{remark}[proposition]{Remark}

\newtheorem{example}[proposition]{Example}
\newtheorem{definition}[proposition]{Definition}

\newcommand{\reg}{{\rm reg}}
\newcommand{\depth}{{\rm depth}}

\tikzstyle{place}=[draw,circle,minimum size=1mm,inner sep=1pt,outer sep=-1.1pt,fill=black]

\usetikzlibrary{shapes.geometric}
\tikzstyle{places}=[draw,rectangle,minimum size=8pt,inner sep=0pt]
\tikzstyle{placesf}=[draw,rectangle,minimum size=5pt,inner sep=0pt]
\tikzstyle{placec}=[draw,circle,minimum size=8pt,inner sep=0pt]
\tikzstyle{placecf}=[draw,circle, minimum size=5pt,inner sep=0pt]

\def\K{\mathbb{K}}
\def\G{\mathcal{G}}
\def\C{\mathcal{C}}
\def\B{\mathcal{B}}

\def\d{\mathrm{depth}}
\def\di{\mathrm{depth}(R/I(G)^{[k]})}
\def\reg{\mathrm{reg}}

\def\l{\langle}
\def\r{\rangle}
\def\x{\mathbf x}
\def\b{\mathbf b}

\begin{document}

\title{Square-free powers of Cohen-Macaulay forests, cycles, and whiskered cycles}

\author{Kanoy Kumar Das}
\address{Chennai Mathematical Institute, India}
\email{kanoydas@cmi.ac.in; kanoydas0296@gmail.com}

\author{Amit Roy}
\address{Chennai Mathematical Institute, India}
\email{amitiisermohali493@gmail.com}

\author{Kamalesh Saha}
\address{Chennai Mathematical Institute, India}
\email{ksaha@cmi.ac.in; kamalesh.saha44@gmail.com}


\vspace*{-0.4cm}
\begin{abstract}
    Let $I(G)^{[k]}$ denote the $k^{th}$ square-free power of the edge ideal $I(G)$ of a graph $G$. In this article, we provide a precise formula for the depth of $I(G)^{[k]}$ when $G$ is a Cohen-Macaulay forest. Using this, we show that for a Cohen-Macaulay forest $G$, the $k^{th}$ square-free power of $I(G)$ is always Cohen-Macaulay, which is quite surprising since all ordinary powers of $I(G)$ can never be Cohen-Macaulay unless $G$ is a disjoint union of edges. Next, we give an exact formula for the regularity and tight bounds on the depth of square-free powers of edge ideals of cycles. In the case of whiskered cycles, we obtain tight bounds on the regularity and depth of square-free powers, which aids in identifying when such ideals have linear resolutions. Additionally, we compute depth of $I(G)^{[2]}$ when $G$ is a cycle or whiskered cycle, and regularity of $I(G)^{[2]}$ when $G$ is a whiskered cycle.
\end{abstract}

\maketitle

\section{Introduction}
Research on edge ideals of graphs and their powers is one of the central topics of study in combinatorial commutative algebra. The notion of edge ideals was introduced by Villarreal \cite{CM} in 1990, and since then, a significant amount of work has been done on this topic. Edge ideals are a specific example of Stanley-Reisner ideals, and because of this, they offer numerous combinatorial insights into graph theory. The study of powers of edge ideals is currently an active area of research but remains quite challenging, which explains its limited investigations so far. Even for some fundamental classes of graphs, such as forests \cite{BHT}, unicyclic graphs \cite{ABS123}, and Cameron-Walker graphs \cite{BBHa2020}, the regularity of powers of edge ideals has been established in recent years. To the best of our knowledge, efforts to determine the depth of all powers of an edge ideal have only been successful for cycles \cite{MiTrVu2023} and Cohen-Macaulay forests \cite{HHT1234}.
\par

Recently, researchers have shown immense interest in studying the square-free powers of edge ideals, which have a nice connection to the matching theory of graphs. Furthermore, the study of square-free powers provides a significant boost in the study of ordinary powers as the multigraded minimal free resolution of the square-free powers is an induced subcomplex of that of the ordinary powers. Let $R=\mathbb{K}[x_1,\ldots,x_n]$ be a polynomial ring with $n$ variables over a field $\mathbb{K}$ and $I\subseteq R$ be a square-free monomial ideal. The $k^{th}$ \textit{square-free power} of $I$, denoted by $I^{[k]}$, is the ideal of $R$ generated by all square-free monomials of the ideal $I^k$. Note that $I^{[k]}=\l 0\r$ for $k>>0$. Let $\mathcal{G}(I)$ denote the unique minimal monomial generating set of a monomial ideal $I$. Then one can observe that $\mathcal{G}(I^{[k]})$ is the set of square-free monomials contained in $\mathcal{G}(I^k)$. Now consider a finite simple graph $G$ on the vertex set $V(G)=\{x_1,\ldots,x_n\}$ and the edge set $E(G)$. A \textit{matching} $M$ of $G$ is a collection of distinct edges of $G$ such that $e\cap f=\emptyset$ for all $e,f\in M$ with $e\neq f$. A matching $M$ is called a \textit{$k$-matching} if $\vert M\vert =k$. For $A\subseteq \{x_1,\ldots,x_n\}$, let us denote by $\mathbf{x}_{A}$ the square-free monomial $\prod_{x_i\in A}x_i$ in $R$. If $I(G)$ denotes the edge ideal of a graph $G$, then
$$I(G)^{[k]}=\left\l \mathbf{x}_{A}\mid A=\bigcup_{e_i\in M} e_i \text{ and } M \text{ is a $k$-matching of $G$}\right\r.$$
In other words, $I(G)^{[k]}$ is generated by products of $k$ minimal generators of $I(G)$ that form a regular sequence in $R$. The ideal $I(G)^{[k]}$ is also known as the $k^{th}$ \textit{matching power} of $G$. Note that if $k=1$, then $I(G)^{[1]}$ is nothing but the edge ideal $I(G)$. Let $\nu(G)$ denote the maximum cardinality of a matching in $G$. Then we have $I(G)^{[k]}\neq \l 0\r$ if and only if $k\leq \nu(G)$. Therefore, in order to study $I(G)^{[k]}$, it is enough to consider $k\leq \nu(G)$.\par 

The study of square-free powers began with the work of Bigdeli, Herzog, and Zaare-Nahandi \cite{BHZN} in 2018. However, it garnered considerable attention after the intriguing work of Erey-Herzog-Hibi-Saeedi Madani \cite{EHHS} in 2022. They gave bounds for the regularity of $I(G)^{[k]}$ and discussed the question of when such powers have a linear resolution or satisfy the square-free Ratliff property. In \cite{EHHM12}, the same authors introduced the normalized depth function on square-free powers and conjectured that such a function is non-increasing. Later, Fakhari disproved this conjecture in \cite{SAFS125}. However, the conjecture is widely open for square-free powers of edge ideals of graphs. Erey-Hibi in \cite{ErHi1} introduced the notion of $k$-admissible matching of a graph and gave a conjecture \cite[Conjecture 31]{ErHi1} on the regularity of square-free powers of forests in terms of the $k$-admissible matching number. Later on, Crupi-Ficarra-Lax \cite{CFL1} settled this conjecture using Betti splitting. For more work on the square-free powers, the reader can look into \cite{FiHeHi,  KaNaQu, SAFS124, Fakhari2024}.

In this paper, we study square-free powers of edge ideals of some fundamental classes of graphs, namely Cohen-Macaulay forests, cycles, and whiskered cycles. Although these graphs may appear simple in structure, the study of their square-free powers is quite challenging, sometimes even more difficult to handle when compared to their ordinary powers. Recall that, Crupi-Ficarra-Lax \cite{CFL1} expressed the regularity of square-free powers of forests in terms of the admissible matching number of the corresponding graph. However, even for a simple enough graph, it is very difficult to find the admissible matching number. In \Cref{sec tree}, we derive a simple numerical formula for the regularity of $k^{th}$ square-free powers of paths (see \Cref{sqf path}) and whiskered paths (see \Cref{proppathwhisk}) in terms of the number of vertices of the underlying graphs and the integer $k$. Our main result in this section is \Cref{thmcmtreedepth}, where we show that for a Cohen-Macaulay forest $G$ of dimension $m$, $\depth(R/I(G)^{[k]})=m+k-1$ for all $1\leq k\leq \nu(G)$. As a consequence of \Cref{thmcmtreedepth}, we obtain the following:
\medskip

\noindent \textbf{\Cref{sqfree powers cm}.} \textit{Let $G$ be a Cohen-Macaulay forest. Then $R/I(G)^{[k]}$ is Cohen-Macaulay for all $k$.
}
\medskip

\noindent The aforementioned result is quite surprising as it follows from \cite[Page 2]{RTY11} and \cite[Corollary 3.10]{HHT1234} that for a graph $G$, $R/I(G)^k$ is Cohen-Macaulay for all $k\geq 1$ if and only if $G$ is a disjoint union of edges (see also \Cref{cm forest all sqf powers}). Thus, it will be interesting to study when the square-free powers of edge ideals of graphs are Cohen-Macaulay. Moving on, in \Cref{sec cycle} of this paper, we provide an explicit formula for the regularity of square-free powers of edge ideals of cycles:
\medskip

 \noindent \textbf{\Cref{cycle reg theorem}.} \textit{Let $C_n$ denote the cycle of length $n$. Then for $1\le k\le \nu(C_n)=\lfloor\frac{n}{2}\rfloor$, we have
    \[
    \reg(I(C_n)^{[k]})=2k+\left\lfloor\frac{n-2k}{3}\right\rfloor.
    \]
    }
    
    \noindent 
    From the above result, one can observe that for a cycle $C_n$, $I(C_n)^{[k]}$ has a linear resolution if and only if $k=\nu(C_n)$ in case of odd $n$, and $k\in\{\nu(C_n),\nu(C_n)-1\}$ in case of even $n$ (\Cref{linear even cycle}). Next, it is known from the work of Jacques \cite{JacquesThesis} that $\depth(R/I(C_n))=\lceil \frac{n-1}{3}\rceil $. For $2\leq k\leq \nu(C_n)$, we show in \Cref{cycle depth lower bound} that $\depth(R/I(C_n)^{[k]})\geq\lceil \frac{n}{3}\rceil+k-1$. As an application, we derive in \Cref{cycle depth 2} that $\depth(R/I(C_n)^{[2]})=\lceil \frac{n}{3}\rceil+1$. It is worthwhile to mention that for any graph $G$, $\depth(R/I(G)^{[k]})=2k-1$ for $k=\nu(G)$ (see \Cref{depth highest power}). For a cycle $C_n$ of even length, we show that $\d(R/I(C_n)^{[k]})=2k-1$ for $k=\nu(C_n)-1=\lfloor\frac{n}{2}\rfloor-1$ (\Cref{prop cycle depth m-1}).

    \Cref{sec whiskered cycle} of this paper deals with the regularity and depth of square-free powers of edge ideals of whiskered cycles. Considering a slightly more general class containing the whiskered cycles, we derive bounds on the regularity of square-free powers of their edge ideals in \Cref{reg whiskeredcycle multiedge}. As a consequence, we obtain
    \medskip
    
    \noindent \textbf{\Cref{whisker reg bound}.} \textit{
Let $W(C_m)$ denote the whisker graph on $C_m$. Then for each $1\le k\le \nu(W(C_m))$, 
$$2k+\left\lfloor\frac{m-k-1}{2}\right\rfloor\le \reg(I(W(C_m))^{[k]})\le 2k+\left\lfloor\frac{m-k}{2}\right\rfloor.$$
    }
    
\noindent
Due to the above-mentioned lower bound, one can see that $I(W(C_m))^{[k]}$ does not have a linear resolution if $k<m-2$ (\Cref{whisker not linear}). Now, it follows from \cite[Proposition 1.1]{MFY} that $\reg(I(W(C_m)))=\lfloor\frac{m}{2}\rfloor+1$. In \Cref{whisker cycle reg 2}, we show that $\reg(I(W(C_m))^{[2]})=3+\left\lfloor\frac{m-1}{2}\right\rfloor$. Regarding the depth, we have $\depth(R/I(W(C_m)))=m$ (see \cite{CM}) and we derive $\depth(R/I(W(C_m))^{[2]})$ as follows:
\medskip

\noindent \textbf{\Cref{whisker cycle depth 2}.} \textit{Let $W(C_m)$ denote the whisker graph on a cycle of length $m$. Then
    \begin{align*}
        \depth(R/I(W(C_m))^{[2]})=
        \begin{cases}
            3 &\text{ if $m=3,$}\\
            m+1 &\text{ if $m>3.$}
        \end{cases}
    \end{align*}
}

\noindent 
Moreover, we show in \Cref{depth whisker upper bound} that $\depth(R/I(W(C_m))^{[k]})\leq m+k-1$ for all $1\leq k\leq m$. Finally, in \Cref{sec conj}, based on our results and computational evidence, we propose three conjectures (Conjecture \ref{conj depth cycle}, \ref{conj whisker reg}, \ref{conj whisker depth}) regarding the depth of square-free powers for cycles and whiskered cycles, and the regularity of square-free powers for whiskered cycles.

\section{Preliminaries and some technical lemmas}

In this section, we briefly recall some relevant results from combinatorics and commutative algebra that will be used throughout the rest of the paper. Moreover, we derive some technical lemmas on square-free powers of edge ideals related to certain ideal theoretic operations.\medskip

\subsection{Graph Theory and Combinatorics:}

    Let $G=(V (G), E(G))$ be a finite simple graph with vertex set $V(G)$ and edge set $E(G)$. For a vertex $x$ in $G$, the set of {\it neighbors} of $x$, denoted by $N_G(x)$, is the set $\{y \in V(G) \mid \{x,y\}\in E(G)\}$. The set of \textit{closed neighbors} of $x$ is the set $N_G(x)\cup \{x\}$, and is denoted by $N_G[x]$. More generally, if $x_1,\ldots,x_t\in V(G)$, then we define $N_G[x_1,\ldots,x_t]=\cup_{i=1}^{t}N_{G}[x_i]$. For $x\in V(G)$, the number $|N_G(x)|$ is called the {\it degree} of $x$, and is denoted by $\deg(x)$. Given a graph $G$, to add a {\it whisker} at a vertex $x$ of $G$, one simply adds a new vertex $y$ and an edge connecting $y$ and $x$ to $G$. For $W\subseteq V(G)$, the graph $G\setminus W$ denotes the graph on the vertex set $V(G)\setminus W$ and the edge set $\{\{a,b\}\in E(G)\mid x\notin \{a,b\}\text{ for each }x\in W\}$. If $W=\{x\}$ for some $x\in V(G)$, then $G\setminus W$ is simply denoted by $G\setminus x$. Similarly, for $W\subseteq V(G)$, the {\it induced subgraph of $G$ on $W$}, denoted by $G[W]$, is the graph $G\setminus (V(G)\setminus W)$. A subset $A\subseteq E(G)$ is called an {\it induced matching} of $G$ if $A$ is a matching in $G$ and $E(G[\cup_{e\in A}e])=A$. The {\it induced matching number} of $G$, denoted by $\nu_1(G)$, is the maximum cardinality of an induced matching in $G$. An induced matching of cardinality $2$ is called a {\it gap} in $G$.
\medskip

\noindent   
\textbf{Various classes of simple graphs:} 
\begin{enumerate}
    \item A {\it path} graph $P_m$ of length $m-1$ is a graph on the vertex set $\{x_1,\ldots,x_{m}\}$ with the edge set $\{\{x_i,x_{i+1}\}\mid 1\le i\le m-1\}$. 
    
    \item A {\it cycle} of length $m$, denoted by $C_m$, is a graph such that $V(C_m)=\{x_1,\ldots,x_m\}$ and $E(C_m)=\{\{x_1,x_m\},\{x_i,x_{i+1}\}\mid 1\le i\le m-1\}$.

    \item A graph $G$ is called a \textit{forest} if it does not contain any induced cycle. A connected forest is known as a \textit{tree}.

    \item A graph is called \textit{chordal} if it does not contain an induced $C_n$ for any $n\geq 4$. A graph whose complement graph is chordal is known as a \textit{co-chordal graph}. 

    \item Let $G$ be a graph with $V(G)=\{x_1,\ldots,x_n\}$ and the edge set $E(G)$. Construct a new graph $W(G)$ with $V(W(G))=\{x_i,y_i\mid 1\le i\le  n\}$ and $E(W(G))=E(G)\cup \{\{x_i,y_i\}\mid 1\le i\le n\}$. Then $W(G)$ is called the \textit{whisker graph} on $G$. The graphs $W(C_m)$ and $W(P_m)$ are called {\it whiskered cycle} and {\it whiskered path}, respectively.
    \end{enumerate}

\subsection{Commutative algebra}
In this subsection, we recall some basic facts related to regularity and depth of graded ideals.\par 

\noindent \textbf{Regularity:} Let $I\subseteq R=\K[x_1,\ldots, x_n]$ be a graded ideal. A graded minimal free resolution of $I$ is an exact sequence
\[
\mathcal F_{\cdot}: \,\, 0\rightarrow F_k\xrightarrow{\partial_{k}}\cdots\xrightarrow{\partial_{2}} F_1\xrightarrow{\partial_1} F_0\xrightarrow{\partial_0} I\rightarrow 0, 
\]
where $F_i=\oplus_{j\in\mathbb N}R(-j)^{\beta_{i,j}(I)}$ for $i\ge 0$, $R(-j)$ is the polynomial ring $R$ with its grading twisted by $j$, and $k\le n$. The numbers $\beta_{i,j}(I)$ are called the $i^{th}$ $\mathbb N$-graded {\it Betti numbers} of $I$ in degree $j$. The {\it Castelnuovo-Mumford regularity} (in short, regularity) of $I$, denoted by $\reg(I)$, is the number $\max\{j-i\mid \beta_{i,j}(I)\neq 0\}$. If $I$ is the zero ideal, then we adopt the convention that $\reg(I)=1$. A graded ideal $I$ with all the minimal generators having the same degree $d$ is said to have a {\it linear resolution} if $\reg(I)=d$. %Moreover, the invariant $\max\{i\mid \beta_{i,j}(R/I)\neq 0\}$ is called the {\it projective dimension} of $R/I$, and is denoted by $\pd(R/I)$. 
\par 

The following are some well-known results regarding the regularity of a graded ideal.

\begin{lemma}\label{reg sum}\cite{RHV}
        Let $I_1\subseteq R_1=\mathbb K[x_1,\ldots,x_m]$ and $I_2\subseteq R_2=\mathbb K[x_{m+1},\ldots, x_n]$ be two graded ideals. Consider the ideal $I=I_1R+I_2R\subseteq R=\mathbb K[x_1,\ldots,x_n]$. Then 
        \[ \reg(I)=\reg(I_1)+\reg(I_2)-1.
        \]
    \end{lemma}

\begin{lemma}\textup{\cite[cf. Lemma 2.10]{DHS}}\label{lemreg}
    Let $I \subseteq R$ be a monomial ideal, $m$ be a monomial of degree $d$ in $R$, and $x$ be an indeterminate in $R$. Then 
\begin{enumerate}
    \item[(i)] $\reg(I+\l x\r)\le \reg(I) $,

    \item[(ii)] $\reg(I:x)\le \reg(I)$,

    \item[(iii)] $\reg I \le \max\{\reg(I : m) + d, \reg( I+ \l m\r)\}$.
\end{enumerate}
    \end{lemma}

We use the following basic result frequently without mentioning it explicitly.

\begin{lemma}\label{colon comma exchange}
    Let $I\subseteq R$ be a monomial ideal, and let $x_1,\ldots,x_r$ be some indeterminates in $R$. Then \[ (( I+\l x_1,\ldots,x_{r-1} \r):x_r)=(I:x_r)+\l x_1,\ldots,x_{r-1} \r.\]
\end{lemma}

The following formulas for the regularity of edge ideals of paths and cycles can be easily derived from the thesis of Jacques \cite{JacquesThesis}.

\begin{lemma}\textup{\cite[cf. Theorem 7.7.34, Theorem 7.6.28]{JacquesThesis}} \label{cycle}
    Let $P_n$ and $C_n$ denote the path and cycle on $n$ vertices, respectively. Then $$\reg(I(P_n))=\reg(I(C_n))=2+\left\lfloor \frac{n-2}{3} \right\rfloor.$$
\end{lemma}

It is a classical result due to Fr\"oberg \cite{Fro} that the edge ideal of a graph $G$ has a linear resolution if and only if $G$ is co-chordal. Later on, Herzog, Hibi, and Zheng \cite{HHZ04} showed that any power of edge ideal of a co-chordal graph also has a linear resolution. Therefore, it follows from \cite[Lemma 1.2]{EHHS} that any square-free power of edge ideal of a co-chordal graph has a linear resolution.
\medskip

\noindent \textbf{Depth:} For the graded ideal $I\subseteq R=\mathbb K[x_1,\ldots,x_n]$, the depth of the finitely generated graded $R$ module $R/I$ is defined as $\d(R/I)=\min \{i\mid H_{\mathfrak{m}}^i(R/I)\neq 0\}$, where $H_{\mathfrak m}^i$ denotes the $i^{th}$ local cohomology module of $R/I$ with respect to the graded maximal ideal
$\mathfrak m=\l x_1,\ldots,x_n\r$. Let $G$ be a graph with $V(G)=\{x_1,\ldots,x_n\}$. Then we define the {\it dimension} of the graph $G$, denoted by $\mathrm{dim}(G)$, to be the Krull dimension of the quotient ring $R/I(G)$. We say that $G$ is {\it Cohen-Macaulay} if $\mathrm{dim}(G)=\d(R/I(G))$. In general, a graded ideal $I\subseteq R$ is called Cohen-Macaulay if $\mathrm{dim}(R/I)=\d(R/I)$.\par 

The following well-known results concerning $\d(R/I)$ for any graded ideal $I$, are essential in proving some of our main theorems in subsequent sections.

\begin{lemma}\label{depth sum}\cite{RHV}
        Let $I_1\subseteq R_1=\mathbb K[x_1,\ldots,x_m]$ and $I_2\subseteq R_2=\mathbb K[x_{m+1},\ldots, x_n]$ be two graded ideals. Consider the ideal $I=I_1R+I_2R\subseteq R=\mathbb K[x_1,\ldots,x_n]$. Then 
        \[ \d(R/I)=\d(R_1/I_1)+\d(R_2/I_2).
        \]
    \end{lemma}
As an immediate consequence of the above lemma, we get the following.
    \begin{lemma}\label{lemdepthextravar}
    Let $J$ be a monomial ideal in $R'=\mathbb{K}[x_1,\ldots,x_m]$, $I=J+\langle x_{m+1},\ldots,x_{r}\rangle$ and $R=\mathbb{K}[x_1,\ldots,x_n]$ with $r\leq n$. Then
    $$\depth(R/I)=\depth(R'/J)+(n-r).$$
\end{lemma}

\begin{lemma}\textup{\cite[cf. Lemma 2.2, Lemma 4.1, Theorem 4.3]{CHHKTT}}\label{lemdepth}
    Let $I$ be a monomial ideal in $R=\mathbb K[x_1,\ldots,x_n]$ and let $f\in R$ be an arbitrary monomial. Then we have the following.
    \begin{enumerate}
        \item[(i)] $\d(R/I)\in\{\d(R/(I+\l f\r)), \d(R/(I:f))\}$,

        \item[(ii)] $\d(R/I)\ge\min\{\d(R/(I+\l f\r)), \d(R/(I:f))\}$,
    
        \item[(iii)] $\d(R/I)\le \d(R/(I:f))$ if $f\notin I$.
    \end{enumerate}
\label{lemdepthequal}
\end{lemma}

\begin{remark}\label{depth highest power}
    Let $G$ be a simple graph. Then it has been proved in \cite[Theorem 1.7]{ErFi12345} that $I(G)^{[\nu(G)]}$ is a polymatroidal ideal (see also \cite[Theorem 1.1]{ErFiForest} for a different proof). Thus, by \cite[Proposition 1.7]{EHHM12}, $\depth(R/I(G)^{[k]})=2k-1$.
\end{remark}

\subsection{Technical lemmas}
In this subsection, we derive some important lemmas, which are used in this paper extensively, and will help in the future study on square-free powers of edge ideals. 

\begin{lemma}\label{lemcolonedge}
    Let $G$ be a graph with vertices $x, y$ such that $\{x,y\}\in E(G)$ and assume that $2\le k\le \nu(G)$.
    \begin{enumerate}[label=(\roman*)]
        \item $ I(G)^{[k]}+\l x \r=I(G')^{[k]}+\l x\r$, where $G'=G\setminus x$.

        \item $(I(G)^{[k]}:x)=\sum_{y\in N_G(x)} yI(G\setminus \{x,y\})^{[k-1]} + I(G\setminus N_G[x])^{[k]}$.

        \item Suppose that $N_G(x)\setminus \{y\}=\{u_1,\ldots ,u_r\}$, $N_G(y)\setminus\{x\}=\{v_1,\ldots ,v_s\}$. For each $1\leq i\leq r$, let $G_{i}$ be the graph on the vertex set $V(G)\setminus\{x,y\}$ and the edge set $E(G_{i})=E(G\setminus\{x,y\})\cup\{\{u_i,v_j\}\mid 1\leq j\leq s, u_i\neq v_j \}$. Then $(I(G)^{[k]}:xy)=\sum_{i=1}^{r}I(G_{i})^{[k-1]}$.
    \end{enumerate}
    \label{sqf colon}
\end{lemma}

\begin{proof}
    (i) This is a straightforward observation.
    
    (ii) It is easy to see that $\sum_{y\in N_G(x)} yI(G\setminus \{x,y\})^{[k-1]} + I(G\setminus N_G[x])^{[k]}\subseteq (I(G)^{[k]}:x)$. Now let $f\in (I(G)^{[k]}:x)$. Then $fx\in I(G)^{[k]}$, and hence there is a $k$-matching in $G$, say $\{\{a_1,b_1\},\ldots , \{a_k,b_k\} \}\subseteq E(G)$ such that $(a_1b_1)\cdots (a_kb_k)\mid fx$. We consider the following two cases.

    \noindent
    \textbf{Case-I:} Let $\{a_t,b_t\}=\{x,y\}$ for some $y\in N_G(x)$, and $1\leq t\leq k$. Then \[y\cdot (a_1b_1)\cdots (a_{t-1}b_{t-1})(a_{t+1}b_{t+1})\cdots (a_kb_k)\mid f ,\] and hence $f\in  yI(G\setminus \{x,y\})^{[k-1]}$.

    \noindent
    \textbf{Case-II:} Assume that $\{a_t,b_t\}\neq \{x,y\}$ for all $y\in N_G(x)$, and for all $1\leq t\leq k$. In that case, if for each $y\in N_G(x)$, $y\neq a_t$, and $y\neq b_t$ for all $1\leq t\leq k$, then $\{a_t,b_t\}\in E(G\setminus N_G[x])$ for all $1\leq t\leq k$. Thus, $f\in I(G\setminus N_G[x])^{[k]}$. Otherwise, if $y=a_t$ for some $1\leq t\leq k$, then $y\cdot (a_1b_1)\cdots (a_{t-1}b_{t-1})(a_{t+1}b_{t+1})\cdots (a_kb_k)\mid f $, and hence $f\in yI(G\setminus \{x,y\})^{[k-1]}$.   

    (iii) Let $f\in \mathcal{G}(I(G_{i})^{[k-1]})$ for some $1\leq i\leq r$. Then there is a $(k-1)$-matching in $G_{i}$, say $\{\{a_1,b_1\},\ldots \{a_{k-1},b_{k-1}\}\}\subseteq E(G_{i})$ such that $f=(a_1b_1)\cdots (a_{k-1}b_{k-1})$. If for all $1\leq t\leq k-1$, $\{a_t,b_t\}\neq \{u_i,v_j\}$, then $f\cdot xy\in I(G)^{[k]}$. Otherwise if for some $1\leq t\leq k-1$, $\{a_t,b_t\}= \{u_i,v_j\}$, then we can write $f\cdot xy=(xu_i)(yv_j)(a_1b_1)\cdots (a_{t-1}b_{t-1})(a_{t+1}b_{t+1})\cdots (a_{k-1}b_{k-1})\in I(G)^{[k]}$. Hence in any case, $f\in (I(G)^{[k]}:xy)$. 

    Conversely, let $g\in (I(G)^{[k]}:xy)$. Then $g\cdot xy\in I(G)^{[k]}$, and hence, there is a $k$-matching in $G$, say $\{\{a_1,b_1\},\ldots \{a_k,b_k\}\}\subseteq E(G)$ such that $(a_1b_1)\cdots (a_{k}b_{k})\mid g\cdot xy$. Again if $\{x,y\}= \{a_t,b_t\}$ for some $1\leq t\leq k$, then $g\in I(G')^{[k-1]}$, where $G'=G\setminus \{x,y\}$. Note that $I(G')^{[k-1]}\subseteq I(G_{i})^{[k-1]}$ for all $1\leq i\leq r$ and thus $g\in \sum_{i=1}^rI(G_{i})^{[k-1]}$. Now, assume that $\{x,y\}\neq \{a_t,b_t\}$ for any $1\leq t\leq k$. If $x,y\notin \{a_i, b_i\mid 1\leq i\leq k\}$, then $g\in I(G\setminus\{x,y\})^{[k]}$ and hence for any $1\leq i\leq r$, $g\in I(G_i)^{[k]}\subseteq \sum_{i=1}^{r}I(G_{i})^{[k-1]}$. Next, if $x\in \{a_i, b_i\mid 1\leq i\leq k\}$ but $y\notin \{a_i, b_i\mid 1\leq i\leq k\}$, then without any loss of generality, assume that $a_1=x$. In that case, $(a_2b_2)(a_3b_3)\cdots (a_kb_k)\mid g $ and hence $g\in I(G\setminus\{x,y\})^{[k-1]}\subseteq \sum_{i=1}^{r}I(G_{i})^{[k-1]}$. Finally if $x,y\in \{a_i, b_i\mid 1\leq i\leq k\}$ and $\{x,y\}\neq \{a_t,b_t\}$ for any $1\leq t\leq k$, without any loss of generality, we can assume that $x=a_1, y=a_2$. Then $b_1\in N_G(x)\setminus\{y\}$, $b_2\in N_G(y)\setminus\{x\}$, and $b_1\neq b_2$. So, $b_1=u_i, b_2=v_j$ for some $1\leq i\leq r,1\leq j\leq s$ and $u_i\neq v_j$. Therefore, $(xu_i)(yv_j)(a_3b_3)\cdots (a_kb_k)\mid g\cdot xy$, which implies $(u_iv_j)(a_3b_3)\cdots (a_kb_k)\mid g $ and hence $g\in I(G_{i})^{[k-1]}$. This completes the proof.
\end{proof}

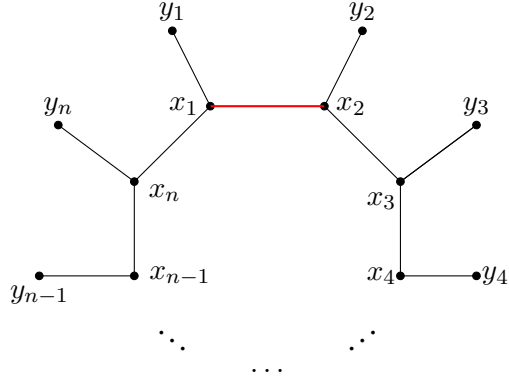
\begin{figure}
    \centering
    \begin{tikzpicture}
        [scale=.50]
        \draw [fill] (0,2.5) circle [radius=0.1]; % x_1
        \draw [fill] (-1,4.5) circle [radius=0.1]; % y_1
        \draw [fill] (5,0.5) circle [radius=0.1]; % x_3
        \draw [fill] (7,2) circle [radius=0.1]; % y_3
        \draw [fill] (5,-2) circle [radius=0.1]; % x_4
        \draw [fill] (7,-2) circle [radius=0.1]; % y_4
        \draw [fill] (3,2.5) circle [radius=0.1]; % x_2
        \draw [fill] (4,4.5) circle [radius=0.1]; % w_1
        \draw [fill] (-2,0.5) circle [radius=0.1];% x_n
        \draw [fill] (-4,2) circle [radius=0.1]; %y_n
        \draw [fill] (-2,-2) circle [radius=0.1];% x_n-1
        \draw [fill] (-4.5,-2) circle [radius=0.1]; %y_n-1
        \node at (-4,2.5) {$y_n$};
        \node at (-1.2,0.2) {$x_{n}$};
        \node at (-0.7,2.5) {$x_1$};
        \node at (-1,5) {$y_1$};
        \node at (4.5,0) {$x_{3}$};
        \node at (3.7,2.5) {$x_2$};
        \node at (4,5) {$y_2$};
        %\node at (1.5,-6) {$G$};
        \node at (7,2.5) {$y_3$};
        \node at (4.5,-2) {$x_{4}$};
        \node at (7.5,-2) {$y_4$};
        \node at (-0.8,-2) {$x_{n-1}$};
        \node at (-4.5,-2.5) {$y_{n-1}$};
        \node at (-1,-3.5) {$\ddots$};
        \node at (4,-3.5) {\reflectbox{$\ddots$}};
        \node at (1.5,-4.5) {$\hdots$};
        \draw [thick,red] (0,2.5)--(3,2.5);
        \draw (0,2.5)--(-2,0.5);
        \draw (4,4.5)--(3,2.5);
        \draw (3,2.5)--(5,0.5);
        \draw (0,2.5)--(-1,4.5);
        \draw (-4,2)--(-2,0.5);
        \draw (7,2)--(5,0.5);
        \draw (-2,-2)--(-2,0.5);
        \draw (-2,-2)--(-4.5,-2);
        \draw (7,2)--(5,0.5);
        \draw (5,0.5)--(5,-2)--(7,-2);
    \end{tikzpicture}
    \caption{A whiskered cycle $G$}
    \label{fig:enter-label0}
\end{figure}

\begin{figure}
    \centering
    \begin{tikzpicture}
        [scale=.45]
        \draw [fill] (0,2.5) circle [radius=0.1]; % x_1
        \draw [fill] (-1,4.5) circle [radius=0.1]; % y_1
        \draw [fill] (5,0.5) circle [radius=0.1]; % x_3
        \draw [fill] (7,2) circle [radius=0.1]; % y_3
        \draw [fill] (5,-2) circle [radius=0.1]; % x_4
        \draw [fill] (7,-2) circle [radius=0.1]; % y_4
        %\draw [fill] (3,2.5) circle [radius=0.1]; % x_2
        %\draw [fill] (4,4.5) circle [radius=0.1]; % y_2
        \draw [fill] (-2,0.5) circle [radius=0.1];% x_n
        \draw [fill] (-4,2) circle [radius=0.1]; %y_n
        \draw [fill] (-2,-2) circle [radius=0.1];% x_n-1
        \draw [fill] (-4.5,-2) circle [radius=0.1]; %y_n-1
        \node at (-4,2.5) {$y_n$};
        \node at (-1.3,0.2) {$x_{n}$};
        \node at (-0.7,2.5) {$y_2$};
        \node at (-1,5) {$y_1$};
        \node at (4.5,0) {$x_{3}$};
        %\node at (3.5,2.5) {$x_2$};
        %\node at (4,5) {$y_2$};
        \node at (7,2.5) {$y_3$};
        \node at (4.5,-2) {$x_{4}$};
        \node at (7.5,-2) {$y_4$};
        \node at (-0.8,-2) {$x_{n-1}$};
        \node at (-4.5,-2.5) {$y_{n-1}$};
        \node at (-1,-3.5) {$\ddots$};
        \node at (4,-3.5) {\reflectbox{$\ddots$}};
        \node at (1.5,-4.5) {$\hdots$};
        %\draw [thick,red] (0,2.5)--(3,2.5);
        \draw [thick,red] (0,2.5)--(-2,0.5);
        %\draw (4,4.5)--(3,2.5);
        %\draw (3,2.5)--(5,0.5);
        \draw [thick,red] (0,2.5)--(-1,4.5);
        \draw  (-4,2)--(-2,0.5);
        \draw (7,2)--(5,0.5);
        \draw (-2,-2)--(-2,0.5);
        \draw (-2,-2)--(-4.5,-2);
        \draw (7,2)--(5,0.5);
        \draw (5,0.5)--(5,-2)--(7,-2);
    \end{tikzpicture}
    \hspace{2em}
    \begin{tikzpicture}
        [scale=.45]
        %\draw [fill] (0,2.5) circle [radius=0.1]; % x_1
        %\draw [fill] (-1,4.5) circle [radius=0.1]; % y_1
        \draw [fill] (5,0.5) circle [radius=0.1]; % x_3
        \draw [fill] (7,2) circle [radius=0.1]; % y_3
        \draw [fill] (5,-2) circle [radius=0.1]; % x_4
        \draw [fill] (7,-2) circle [radius=0.1]; % y_4
        %\draw [fill] (3,2.5) circle [radius=0.1]; % x_2
        \draw [fill] (4,3.5) circle [radius=0.1]; % y_2
        \draw [fill] (-2,0.5) circle [radius=0.1];% x_n
        \draw [fill] (-4,2) circle [radius=0.1]; %y_n
        \draw [fill] (-2,-2) circle [radius=0.1];% x_n-1
        \draw [fill] (-4.5,-2) circle [radius=0.1]; %y_n-1
        \node at (-4,2.5) {$y_n$};
        \node at (-1.3,0) {$x_{n}$};
        %\node at (-0.5,2.5) {$x_1$};
        %\node at (-1,5) {$y_1$};
        \node at (4.5,0) {$x_{3}$};
        %\node at (3.5,2.5) {$x_2$};
        \node at (4,4) {$y_1$};
        \node at (7,2.5) {$y_3$};
        \node at (4.5,-2) {$x_{4}$};
        \node at (7.5,-2) {$y_4$};
        \node at (-0.8,-2) {$x_{n-1}$};
        \node at (-4.5,-2.5) {$y_{n-1}$};
        \node at (-1,-3.5) {$\ddots$};
        \node at (4,-3.5) {\reflectbox{$\ddots$}};
        \node at (1.5,-4.5) {$\hdots$};
        \draw  [thick,red] (5,0.5)--(-2,0.5);
        %\draw (0,2.5)--(-2,0.5);
        %\draw (4,4.5)--(3,2.5);
        %\draw (3,2.5)--(5,0.5);
        \draw [thick,red] (5,0.5)--(4,3.5);
        \draw (-4,2)--(-2,0.5);
        \draw (7,2)--(5,0.5);
        \draw (-2,-2)--(-2,0.5);
        \draw (-2,-2)--(-4.5,-2);
        \draw (7,2)--(5,0.5);
        \draw (5,0.5)--(5,-2)--(7,-2);
    \end{tikzpicture}
    \caption{The graphs $ G_1$ and $G_2$}
    \label{fig:enter-label1}
\end{figure}
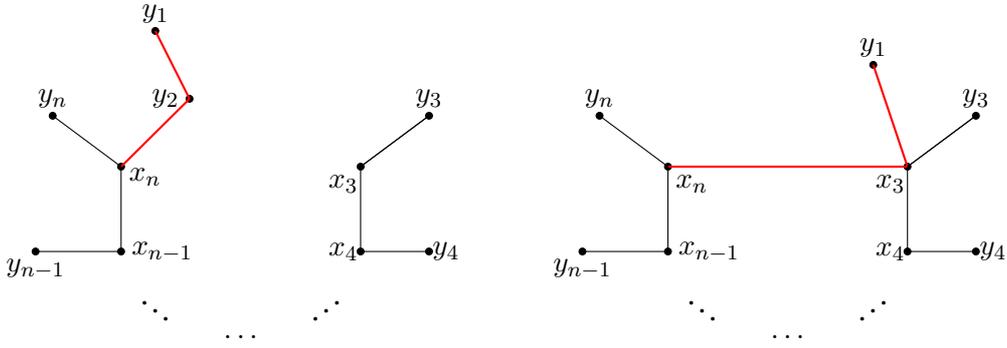

\begin{remark}
    Following the proof of \Cref{lemcolonedge}(iii) it is easy to observe that $(I(G)^{[k]}:xy)$ can also be expressed in the following alternative way. For each $1\le j\le s$, let $H_j$ be the graph on the vertex set $V(G)\setminus\{x,y\}$, and edge set $E(H_{j})=E(G\setminus\{x,y\})\cup\{\{u_i,v_j\}\mid 1\leq i\leq r, u_i\neq v_j \}$. Then $(I(G)^{[k]}:xy)=\sum_{j=1}^{s}I(H_{j})^{[k-1]}$. 
\end{remark}

\begin{example}
  Consider the graphs $G$, $G_1$ and $G_2$ as in \Cref{fig:enter-label0} and \Cref{fig:enter-label1}. Then, applying \Cref{lemcolonedge}(iii), we have $(I(G)^{[k]}:x_1x_2)=I(G_1)^{[k-1]}+I(G_2)^{[k-1]}$ for $1\leq k\leq \nu(G)$.
\end{example}

\begin{remark}\label{second power colon} 
    Note that if $k=2$ in \Cref{lemcolonedge}(iii), then $(I(G)^{[2]}:xy)=I(G')$, where $V(G')=V(G)\setminus \{x,y\}$, and $E(G')=\cup_{i=1}^rE(G_i)$. In other words, $E(G')=E(G\setminus \{x,y\})\cup\{\{a,b\}\mid a\in N_G(x),b\in N_G(y), a\neq b,a,b\notin\{x,y\}\}$. This expression of $(I(G)^{[2]}:xy)$ also appeared in \cite[Lemma 2.6]{EHHS}. 
\end{remark}

\begin{lemma}\label{colon with variable}
    Let $G$ be a graph and $x\in V(G)$ such that $N_G(x)=\{y_1,\ldots,y_n\}$. Then for all $1\le k\le \nu(G)$, $((I(G)^{[k]}+\l xy_1,\ldots,xy_n\r):x)=I(G\setminus\{x,y_1,\ldots,y_n\})^{[k]}+\l y_1,\ldots,y_n\r$.
\end{lemma}
\begin{proof}
The proof of this lemma easily follows from \Cref{lemcolonedge}(ii) and the fact that $((I(G)^{[k]}+\l xy_1,\ldots,xy_n\r):x)=(I(G)^{[k]}:x)+\l y_1,\ldots,y_n\r$.
\end{proof}

\noindent \textbf{Note:} Whenever we consider the edge ideal $I(G)$ of a graph $G$, we write $R$ to mean the corresponding ambient ring $\mathbb{K}[x\mid x\in V(G)]$.

\section{Square-free powers of forests}\label{sec tree}

In this section, we study the regularity and depth of square-free powers of edge ideals of various classes of trees. More precisely, we provide an explicit formula for $\reg(I(G)^{[k]})$ in terms of $n$ and $k$, where $G$ is either a $P_n$ or it is obtained from $P_n$ by attaching multiple whiskers at each vertex of $P_n$. Furthermore, deriving the depth of $R/I(G)^{[k]}$ for a Cohen-Macaulay forest $G$, we show that $R/I(G)^{[k]}$ is Cohen-Macaulay for all $k\geq 1$. 

\medskip

The notion of admissible matching of a graph in the context of square-free powers of edge ideals has been introduced by Erey and Hibi in \cite{ErHi1}. The admissible matching number turns out to be a perfect combinatorial invariant to estimate the regularity of square-free powers of forests (see also \cite{CFL1}). Let $k$ and $n$ be two positive integers. We say that a sequence $(a_1,\ldots ,a_n)$ of integers is $k$-admissible if $a_i\geq 1$ for all $1\leq i\leq n$, and $\sum_{i=1}^na_i\leq n+k-1$.

\begin{definition}\cite[Definition 12]{ErHi1}
    Let $G$ be a graph with matching number $\nu(G)$. Let $M$ be a matching of $G$. Then for any $1\leq k\leq \nu(G)$, we say that $M$ is a $k$-admissible matching if there exists a sequence $M_1,\ldots ,M_r$ of non-empty subsets of $M$ such that 
    \begin{enumerate}
        \item $M=M_1\cup \cdots \cup M_r$,
        \item $M_i\cap M_j=\emptyset$ for all $i\neq j$,
        \item for all $i\neq j$, if $e_i\in M_i$ and $e_j\in M_j$, then $\{e_i,e_j\}$ is a gap in $G$,
        \item the sequence $(|M_1|,\ldots , |M_r|)$ is $k$-admissible,
        \item the induced subgraph of $G$ on $\cup_{e\in M_i}e$ is a forest for all $1\leq i\leq r$.
    \end{enumerate}
\end{definition}

\begin{definition}\cite[Definition 13]{ErHi1}\label{admissible matching definition}
    The $k$-admissible matching number of a graph $G$, denoted by $\mathrm{aim}(G, k)$, is defined as
    \[\mathrm{aim}(G, k) := \max\{|M| \mid M \text{ is a $k$-admissible matching of $G$}\}\]
    for $1 \leq k \leq \nu(G)$. If $G$ has no $k$-admissible matching, then $\mathrm{aim}(G, k)$ is defined to be $0$.
\end{definition}

In \cite{ErHi1}, the authors conjectured that if $G$ is any forest, then $\reg(I(G)^{[k]})=\mathrm{aim}(G,k)+k$ for all $1\leq k\leq \nu(G)$. Later, Crupi, Ficarra, and Lax \cite{CFL1} proved this conjecture for forests. However, it is not straightforward to determine the combinatorial invariant $\mathrm{aim}(G,k)$, even for simpler classes of trees like paths, whiskered paths, etc. In the following proposition, we give a numerical formula for the regularity of square-free powers of paths.
\begin{proposition}\label{sqf path}
    Let $P_n$ denote the path graph on $n$ vertices. Then for $1\le k\le \nu(P_n)=\lfloor\frac{n}{2}\rfloor$,
    \[
    \reg(I(P_n)^{[k]})= 2k+\left\lfloor \frac{n-2k}{3} \right\rfloor.
    \]
\end{proposition}
\begin{proof}
        First, we consider the case $k=\lfloor\frac{n}{2}\rfloor$. Then by \cite[Theorem 5.1]{BHZN}, $I(P_n)^{[k]}$ has a linear resolution, and thus, $\reg(I(P_n)^{[k]})= 2k$. Also, if $k=1$, then the formula follows from \Cref{cycle}. Therefore, we may assume that $1<k<\lfloor\frac{n}{2}\rfloor$. Note that, the path graph $P_n$ is an example of a tree. Therefore, by \cite[Theorem 3.6]{CFL1}, $\reg(I(P_n)^{[k]})=\mathrm{aim}(P_n,k)+k$. We first show that $\mathrm{aim}(P_n,k)\ge  k+\lfloor \frac{n-2k}{3} \rfloor$, which will imply $\reg(I(P_n)^{[k]})\ge 2k+\lfloor \frac{n-2k}{3} \rfloor$. Let $E(P_n)=\{\{x_i,x_{i+1}\}\mid 1\le i\le n-1\}$. Consider the set $M'=\{\{x_1,x_2\},\{x_3,x_4\},\ldots,\{x_{2k-1},x_{2k}\}\}$. Now $G\setminus\{x_1,\ldots,x_{2k+1}\}=P_{n-2k-1}$. Therefore, $\nu_1(G\setminus\{x_1,\ldots,x_{2k+1}\})=\lfloor\frac{n-2k}{3}\rfloor$. Let $M$ be an induced matching of $G\setminus\{x_1,\ldots,x_{2k},x_{2k+1}\}$ such that $|M|=\lfloor\frac{n-2k}{3}\rfloor$. Suppose $M=\left\{e_1,\ldots,e_{\lfloor\frac{n-2k}{3}\rfloor}\right\}$. For each $e_i\in M$, let us consider the set $M_i=\{e_i\}$. Then $|M'|+\sum_{i=1}^{\lfloor \frac{n-2k}{3}\rfloor}|M_i|=k+\lfloor \frac{n-2k}{3}\rfloor\le k+(\lfloor \frac{n-2k}{3}\rfloor+1)-1$. Hence, by \Cref{admissible matching definition}, $M'\sqcup\left(\sqcup_{i=1}^{\lfloor \frac{n-2k}{3}\rfloor}M_i\right)$ forms a $k$-admissible matching of $G$. Thus, $\mathrm{aim}(G,k)\ge k+\lfloor \frac{n-2k}{3} \rfloor$.

        Now, we prove that $\reg(I(P_n)^{[k]})\leq 2k+\lfloor \frac{n-2k}{3}\rfloor$ by induction on $n$. Note that when $n\leq 3$, we have the desired formula from \Cref{cycle}. Thus, we assume that $n\geq 4$. First, we consider the ideal $I(P_n)^{[k]}+\langle x_{n-2}x_{n-1}, x_{n-1}x_n\rangle$. Then by \Cref{colon with variable}, $((I(P_n)^{[k]}+\langle x_{n-2}x_{n-1}, x_{n-1}x_n\rangle):x_{n-1})=I(P_n\setminus \{x_{n-2},x_{n-1},x_{n}\} )^{[k]}+\langle x_{n-2},x_{n}\rangle=I(P_{n-3})^{[k]}+\langle x_{n-2},x_{n}\rangle$. Therefore, by the induction hypothesis, $\reg((I(P_n)^{[k]}+\langle x_{n-2}x_{n-1}, x_{n-1}x_n\rangle):x_{n-1})=2k+\lfloor \frac{n-3-2k}{3}\rfloor=2k+\lfloor \frac{n-2k}{3}\rfloor-1$. Now by \Cref{lemcolonedge}, $I(P_n)^{[k]}+\langle x_{n-2}x_{n-1}, x_{n-1}x_n, x_{n-1}\rangle=I(P_{n-2})^{[k]}+\langle x_{n-1}\rangle$, and thus, $\reg(I(P_n)^{[k]}+\langle x_{n-2}x_{n-1}, x_{n-1}x_n, x_{n-1}\rangle)=2k+\lfloor \frac{n-2-2k}{3}\rfloor\leq 2k+\lfloor \frac{n-2k}{3}\rfloor$. Hence, by \Cref{lemreg}(iii), we get $\reg(I(P_n)^{[k]}+\langle x_{n-2}x_{n-1}, x_{n-1}x_n\rangle)\leq 2k+\lfloor \frac{n-2k}{3}\rfloor$. Now let us consider the ideal $((I(P_n)^{[k]}+\langle x_{n-2}x_{n-1}\rangle) : x_{n-1}x_n)=\langle x_{n-2}\rangle + (I(P_n\setminus x_{n-2})^{[k]} : x_{n-1}x_n)$. By \Cref{colon comma exchange} \& \ref{lemcolonedge}, $((I(P_n)^{[k]}+\langle x_{n-2}x_{n-1}\rangle) : x_{n-1}x_n)=\langle x_{n-2}\rangle + I(P_{n}\setminus\{x_{n-2},x_{n-1},x_n\})^{[k-1]}=\langle x_{n-2}\rangle + I(P_{n-3})^{[k-1]}$. Again by the induction hypothesis, we have $\reg((I(P_n)^{[k]}+\langle x_{n-2}x_{n-1}\rangle) : x_{n-1}x_n)=2k-2+\lfloor \frac{n-3-2k+2}{3}\rfloor\leq 2k-2+\lfloor \frac{n-2k}{3}\rfloor$. Thus, by \Cref{lemreg}(iii), we obtain $\reg(I(P_n)^{[k]}+\langle x_{n-2}x_{n-1}\rangle)\leq 2k+\lfloor \frac{n-2k}{3}\rfloor$. Finally, $(I(P_n)^{[k]}: x_{n-2}x_{n-1})=I(G_1)^{[k-1]}$, where $G_1\cong P_{n-2}$ by \Cref{sqf colon}(iii). Hence, $\reg(I(P_n)^{[k]}: x_{n-2}x_{n-1})=2k-2+\lfloor \frac{n-2-2k+2}{3}\rfloor=2k-2+\lfloor \frac{n-2k}{3}\rfloor$. Therefore, by \Cref{lemreg}(iii), we obtain $\reg(I(P_n)^{[k]})\leq 2k+\lfloor \frac{n-2k}{3}\rfloor$, which completes the proof.
\end{proof}

Next, we compute the regularity of square-free powers of a class of trees, which contain the class of whiskered paths. More specifically, we consider the graphs which are obtained from path graphs by attaching multiple whiskers on each vertex of the corresponding path. To begin with, we fix the following notation. For positive integers $m,r_1,\ldots,r_m$, let $G_{1,m,r_1,\ldots,r_m}$ denote the graph on the vertex set $\{x_i,y_{i,j}\mid 1\le i\le m, 1\le j\le r_i\}$ with the edge set
    \[
    E(G_{1,m,r_1,\ldots,r_m})=\{\{x_i,x_{i+1}\},\{x_i,y_{i,j}\},\{x_m,y_{m,l}\}\mid 1\le i\le m-1,1\le j\le r_i,1\le l\le r_m\}.
    \]
Based on these notations, we define the following class of graphs:
\[
\overline{W}(P_m):=\{G\mid G\cong G_{1,m,r_1,\ldots,r_m}\text{ for some }r_1,\ldots,r_m\ge 1\}.
\]

\begin{proposition}\label{proppathwhisk}
    Let $G\in \overline{W}(P_m)$. Then $\reg(I(G)^{[k]})=2k+\lfloor \frac{m-k}{2} \rfloor$ for each $1\le k\le\nu(G)= m$. 
\end{proposition}

\begin{proof}
    The proof is similar to that of \Cref{sqf path}. First, note that if $k=m$, then $I(G)^{[k]}$ has a linear resolution by \cite[Theorem 5.1]{BHZN}, and thus, $\reg(I(G)^{[k]})= 2k$. Also, if $k=1$, then the formula follows from \cite[Corollary 2.13]{Zheng} since $G$ is a tree and $\nu_1(G)=2+\lfloor\frac{m-1}{2}\rfloor$. Therefore, we may assume that $1<k<m$. Now, by \cite[Theorem 3.6]{CFL1}, $\reg(I(G)^{[k]})=\mathrm{aim}(G,k)+k$. We first show that $\mathrm{aim}(G,k)\ge  k+\lfloor \frac{m-k}{2} \rfloor$, which gives $\reg(I(G)^{[k]})\ge 2k+\lfloor \frac{m-k}{2} \rfloor$. Without loss of generality, let $G=G_{1,m,r_1,\ldots,r_m}$ for some positive integers $r_1,\ldots,r_m$. Consider the set $M'=\{\{x_1,y_{1,1}\},\ldots,\{x_k,y_{k,1}\}\}$. Now $G\setminus\{x_1,\ldots,x_{k+1}\}\in \overline{W}(P_{m-k-1})$. Hence, it is easy to see that $\nu_1(G\setminus\{x_1,\ldots,x_{k+1}\})=\lfloor\frac{m-k}{2}\rfloor$. Let $M$ be an induced matching of $G\setminus\{x_1,\ldots,x_{k+1}\}$ such that $|M|=\lfloor\frac{m-k}{2}\rfloor$. Suppose $M=\left\{e_1,\ldots,e_{\lfloor\frac{m-k}{2}\rfloor}\right\}$. For each $e_i\in M$, let us consider the set $M_i=\{e_i\}$. Then $|M'|+\sum_{i=1}^{\lfloor \frac{m-k}{2}\rfloor}|M_i|=\lfloor \frac{m-k}{2}\rfloor+k\le k+(\lfloor \frac{m-k}{2}\rfloor+1)-1$. Hence, by \Cref{admissible matching definition}, $M'\sqcup\left(\sqcup_{i=1}^{\lfloor \frac{m-k}{2}\rfloor}M_i\right)$ forms a $k$-admissible matching of $G$. Thus, $\mathrm{aim}(G,k)\ge k+\lfloor \frac{m-k}{2} \rfloor$.
    \par 

Now, we will show $\reg(I(G)^{[k]})\leq 2k+\lfloor \frac{m-k}{2}\rfloor$ by induction on $m$. If $m=1$, then $k=m=1$, and $I(G)^{[1]}=I(G)$ has a linear resolution as $G^c$ is a co-chordal graph. Hence, $\reg (I(G)^{[1]})=2$. So, let us assume $m\geq 2$. First, consider the ideal $L=I(G)^{[k]}+\langle x_1y_{1,i}, x_1x_2\mid 1\leq i\leq r_1\rangle$. Note that $(L:x_1)=\langle x_2, y_{1,i} \mid 1\leq i\leq r_1\rangle + I(G_1)^{[k]}$, where $G_1=G\setminus \{x_1,x_2,y_{1,i} \mid 1\leq i\leq r_1\}$ (by \Cref{lemcolonedge}). Observe that $G_1\in \overline{W}(P_{m-2})$, and if $k=m$ or $k=m-1$, then $I(G_1)^{[k]}=\l 0\r$. Hence, by the induction hypothesis, we get $\reg(L:x_1)=\reg(I(G_1)^{[k]})\leq 2k+\lfloor \frac{m-2-k}{2}\rfloor= 2k-1+\lfloor \frac{m-k}{2}\rfloor$. Also, by \Cref{lemcolonedge}, $L+\langle x_1\rangle= I(G_2)^{[k]}+\langle x_1\rangle$, where $G_2=G\setminus \{x_1\}$, and $I(G_2)^{[k]}=\l 0\r$, when $k=m$. Thus, by the induction hypothesis, $\reg(L+\langle x_1\rangle)=\reg(I(G_2))\leq 2k+\lfloor \frac{m-1-k}{2}\rfloor \leq 2k+\lfloor \frac{m-k}{2}\rfloor$. Consequently, by applying \Cref{lemreg}, we obtain $\reg(L)\leq 2k+\lfloor \frac{m-k}{2}\rfloor$. Next, we consider the ideal 
    \[
    L_1=I(G)^{[k]}+\langle x_1y_{1,i}\mid 1\leq i\leq r_1\rangle.
    \] 
    Note that $L_1+\l x_1x_2\r=L$. Now, by \Cref{lemcolonedge}, $(L_1:x_1x_2)=\langle y_{1,i} \mid 1\leq i\leq r_1\rangle + (I(G)^{[k]}:x_1x_2)=\langle y_{1,i} \mid 1\leq i\leq r_1\rangle + I(G_1)^{[k-1]}$. Therefore, by the induction hypothesis, $\reg(L_1:x_1x_2)= \reg(I(G_1)^{[k-1]})\leq  2k-2+\lfloor \frac{m-2-k}{2}\rfloor\leq 2k-2+\lfloor \frac{m-k}{2}\rfloor$. Hence, applying \Cref{lemreg} again, we obtain 
    \[
    \reg(L_1)\leq 2k+\left\lfloor \frac{m-k}{2}\right\rfloor.
    \]
Next, for $1\le i\le r_1$ consider the ideal $L_{2,i}=I(G)^{[k]}+\l x_1y_{1j}\mid 1\le j\le i-1\r$. Note that $L_{2,1}=I(G)^{[k]}$ and $L_{2,r_1}+\l x_1y_{1r_1}\r=L_1$. Moreover, for each $1\le i\le r_1-1$, $L_{2,i}+\l x_1y_{1i}\r=L_{2,i+1}$. Again, for $1\le i\le r_1$, $(L_{2,i}:x_1y_{1i})=\l y_{1,j}\mid 1\le j\le i-1\r+I(G\setminus x_1)^{[k-1]}$ (by \Cref{lemcolonedge}). Since $G\setminus x_1=P_{m-1}$, using the induction hypothesis and repeatedly applying \Cref{lemreg}, we obtain 
$\reg(I(G)^{[k]})=\reg(L_{2,1})\le 2k+\lfloor\frac{m-k}{2}\rfloor.$ 
\end{proof}

The next theorem and the corollary afterwards are the main results of this section, which deal with the depth and Cohen-Macaulay property of square-free powers of edge ideals of Cohen-Macaulay forests. Note that, if $G$ is a Cohen-Macaulay forest of dimension $m$, then by \cite[cf. Theorem 2.4]{CM}, $G=W(T_m)$ for some forest $T_m$ on $m$ number of vertices. Thus, for such a graph $G$, we have $\vert V(G)\vert =2m$ and $\nu(G)=m$.

   \begin{theorem}\label{thmcmtreedepth}
    If $G$ is a Cohen-Macaulay forest of dimension $m$, then we have
    \[
    \depth(R/I(G)^{[k]})=m+k-1,
    \]
    for all $1\leq k\leq \nu(G)=m$.
\end{theorem}
\begin{proof}
    Let $G=W(T_m)$, where $T_m$ is a forest on the vertex set $V(T_m)=\{x_1,\ldots,x_m\}$. More precisely, let $V(G)=V(T_m)\cup \{y_1,\ldots,y_m\}$, and $E(G)=E(T_m)\cup\{\{x_i,y_i\}\mid 1\le i\le m\}$. Then it is easy to see that $\nu(G)=m$. We prove the formula of $\di$ by induction on $m$. If $m=1$, then $G$ has only one edge $\{x_1,y_1\}$. In this case, $\nu(G)=1$ and $x_1-y_1$ is a regular element on $R/I(G)^{[1]}$. Therefore, $\depth(R/I(G)^{[k]})=1$ as $\mathrm{dim}(R/I(G)^{[1]})=1$. Now, let us assume $m\geq 2$. If $k=1$, then there is nothing to prove. Also, if $k=m$, then $I(G)^{[m]}$ has only one generator $\prod_{i=1}^mx_iy_i$. Note that $\di\ge 2m-1$ (by \cite[cf. Proposition 1.1]{EHHM12}), and since $\mathrm{dim}(G)=2m-1$, we have $\di=2m-1=m+m-1$. So, we may also assume $2\leq k\leq m-1$. We break the proof into two cases:\par 

    \noindent \textbf{Case-I:} Let $G$ consist of a disjoint union of edges, i.e., $\deg(x_i)=1$ for each $1\le i\le m$. By \Cref{lemcolonedge}, we have
\[
(I(G)^{[k]}:x_1y_1)=I(G\setminus\{x_1,y_1\})^{[k-1]}.
    \]
    Note that $k-1\le m-1=\nu(G\setminus\{x_1,y_1\})$. Thus, by the induction hypothesis and \Cref{lemdepthextravar}, 
    \begin{align}\label{eq3}
        \depth(R/(I(G)^{[k]}:x_1y_1))=(m-1)+(k-1)-1+2=m+k-1.
    \end{align}
    Now, consider the ideal $J=I(G)^{[k]}+ \l x_1y_1\r=I(G\setminus \{x_1,y_1\})^{[k]}+\l x_1y_1\r$. Then $J+\l x_1\r=I(G\setminus\{x_1,y_1\})^{[k]}+\l x_1\r$ and $(J:x_1)=I(G\setminus\{x_1,y_1\})^{[k]}+\l y_1\r$. Again note that $k\le m-1$ and thus using the induction hypothesis and \Cref{lemdepthextravar}, we get
    \begin{align}\label{eq4}
        \d(R/J+\l x_1\r)=(m-1)+k-1+1=m+k-1,
    \end{align}
    \begin{align}\label{eq5}
        \d(R/(J:x_1))=(m-1)+k-1+1=m+k-1.
    \end{align}
    Therefore, by Equations \eqref{eq4}, \eqref{eq5} and \Cref{lemdepthequal}, we have
    \begin{align}\label{eq6}
        \d(R/J)=\d(R/I(G)^{[k]}+\l x_1y_1\r)=m+k-1.
    \end{align}
    Finally, using Equations \eqref{eq3}, \eqref{eq6} and \Cref{lemdepthequal}, we obtain
    $$\d(R/I(G)^{[k]})=m+k-1.$$
    \noindent \textbf{Case-II:} Let $G$ contain some $x_i$ which has degree at least $2$. Then $T_m$ is a non-empty forest and hence contains a leaf, say $x_1\in V(T_m)$. Let $x_2$ be the unique neighbor of $x_1$ in $T_m$. Then $N_{G}(x_1)=\{x_2,y_1\}$. proceeding as in Case-I, in this case too, we get 
    \begin{align}\label{eq7}
        \depth(R/(I(G)^{[k]}:x_1y_1))=m+k-1.
    \end{align}
    Now we proceed to show that $\d(R/J')=m+k-1$, where $J'=I(G)^{[k]}+\l x_1y_1,x_1x_2\r$. Observe that $J'+\l x_1\r=I(G\setminus\{x_1,y_1\})^{[k]}+\l x_1\r$, and $(J':x_1)=I(G\setminus \{x_1,y_1,x_2,y_2\})^{[k]}+\l x_2,y_1\r$. Now if $k=m-1$, then $I(G\setminus \{x_1,y_1,x_2,y_2\})^{[k]}$ is the zero ideal and in that case $\d(R/(J':x_1))=2m-2$. Also, for $k=m-1$ we have, $I(G\setminus \{x_1,y_1\})^{[k]}=\prod_{i=2}^mx_iy_i$ and thus using \Cref{lemdepthextravar} we obtain $\d(R/J'+\l x_1\r)=2m-2$. Therefore, using \Cref{lemdepthequal} we get $\d(R/J')=2m-2=m+k-1$, where $k=m-1$. Now if $k\le m-2$, then $k\le \min\{\nu(G\setminus\{x_1,y_1\}),\nu(G\setminus\{x_1,y_1,x_2,y_2\})\}=m-2$. Hence, using the induction hypothesis and \Cref{lemdepthextravar}, we obtain
    \begin{align*}
        \d(R/J'+\l x_1\r)=\d(R/I(G\setminus\{x_1,y_1\})^{[k]}+\l x_1\r)=m+k-1,
    \end{align*}
    \begin{align*}
        \d(R/(J':x_1))&=\d(R/I(G\setminus \{x_1,y_1,x_2,y_2\})^{[k]}+\l x_2,y_1\r)\\
        &=(m-2)+k-1+2\\
        &=m+k-1.
    \end{align*}
    Due to the above two equations and \Cref{lemdepthequal}, for each $k\le m-1$, we have
    \begin{align}\label{eq8}
        \d(R/I(G)^{[k]}+\l x_1y_1,x_1x_2\r)=m+k-1.
    \end{align}
    Now, applying \Cref{lemcolonedge}, one can observe that 
    $$((I(G)^{[k]}+\l x_1y_1\r):x_1x_2)=I(G\setminus\{x_1,y_1,x_2,y_2\})^{[k-1]}+\l y_1\r.$$
    Thus, by the induction hypothesis and \Cref{lemdepthequal}, we have
    \begin{align}\label{eq9}
        \d(R/((I(G)^{[k]}+\l x_1y_1\r):x_1x_2))=(m-2)+(k-1)-1+3=m+k-1.
    \end{align}
    Now, using Equations \eqref{eq8}, \eqref{eq9} and \Cref{lemdepthequal}, we get
    \begin{align}\label{eq10}
        \d(R/I(G)^{[k]}+\l x_1y_1\r)=m+k-1.
    \end{align}
    Hence, from equations \eqref{eq7}, \eqref{eq10} and \Cref{lemdepthequal}, it follows that $\d(R/I(G)^{[k]})=m+k-1$.
\end{proof}

\begin{corollary}\label{sqfree powers cm}
    Let $G$ be a Cohen-Macaulay forest. Then $R/I(G)^{[k]}$ is Cohen-Macaulay for all $k$. 
\end{corollary}

\begin{proof}
     Note that, if $k>\nu(G)$, then $I(G)^{[k]}$ is the zero ideal and thus $R/I(G)^{[k]}$ is Cohen-Macaulay. Therefore, we may assume that $1\le k\le \nu(G)$. As before, let $G=W(T_m)$ for some forest $T_m$, where $V(T_m)=\{x_1,\ldots,x_m\}$. Thus $V(G)=\{y_1,\ldots,y_m\}\cup V(T_m)$ and $E(G)=E(T_m)\cup\{\{x_i,y_i\}\mid 1\le i\le m\}$. Let us take a subset $S\subseteq V(G)$ such that $\vert S\vert =m-k$. Let $\mathfrak{p}$ be the prime ideal generated by $S$, i.e., $\mathfrak{p}=\l S\r $. Corresponding to $S$, let us consider the set of indices $\mathcal A_{S}$ as follows:
    \[
    \mathcal A_{S}:=\{1\le i\le m\mid x_i\in S \text{ or }y_i\in S\}.
    \]
    Since $\vert S\vert=m-k$, we have $|\mathcal A_S|\leq m-k$. Thus, there exist at least $k$ integers between $1$ to $m$, which do not belong to $\mathcal A_S$. Without loss of generality, let $1,\ldots,k\not\in \mathcal A_S$. Then 
    \[
    \{x_1,\ldots,x_k,y_1,\ldots,y_k\}\cap S=\emptyset.
    \]
    Now, $x_1\cdots x_ky_1\cdots y_k$ is a minimal generator of $I(G)^{[k]}$, which does not belong to $\mathfrak{p}$. Thus, $I(G)^{[k]}\not\subset\mathfrak{p}$. Since $S$ was chosen arbitrarily, we can say that there exists no prime ideal of height $\leq m-k$, which contains the ideal $I(G)^{[k]}$. Therefore, $\mathrm{height}(I(G)^{[k]})\geq m-k+1$, which implies $\mathrm{dim}(R/I(G)^{[k]})\leq 2m-(m-k+1)=m+k-1$. Hence, it follows from \Cref{thmcmtreedepth} that $R/I(G)^{[k]}$ is Cohen-Macaulay with dimension $m+k-1$. 
\end{proof}

\begin{remark}\label{cm forest all sqf powers}
    It is well-known that for a square-free monomial ideal $I\subseteq R$, the minimal primes of $I$ and $I^k$ are the same, and thus, $\dim(R/I)=\dim(R/I^k)$. However, it is very rare when some powers of an edge ideal become Cohen-Macaulay. In this context, note that, if $G$ is a Cohen-Macaulay forest of dimension $m$, then by \cite[Corollary 3.10]{HHT1234}, we have $\depth(R/I(G)^k)=\max\{m-k+1,r\}$, where $r$ is the number of connected components of $G$. Based on this, one can observe that if $I(G)$ is not a complete intersection ideal (in other words, $G$ is not a disjoint union of edges), then we have $\depth(R/I(G)^{k})<m=\dim(R/I(G)^k)$. In other words, for a Cohen-Macaulay tree $G$ which is not a disjoint union of edges, $R/I(G)^k$ is never Cohen-Macaulay for $k\geq 2$, whereas we have proved in \Cref{sqfree powers cm} that $R/I(G)^{[k]}$ is always Cohen-Macaulay, which is quite surprising.
\end{remark}

\section{Square-free powers of cycles}\label{sec cycle}

In this section, we obtain a formula for the regularity and bounds on the depth of square-free powers of edge ideals of cycles. Although the regularity of ordinary powers of edge ideals of cycles was determined in 2015 \cite{BHT}, an explicit formula for the depth of ordinary powers of cycles has been established very recently in \cite{MiTrVu2023}. Using our bounds on the depth of square-free powers of edge ideals of cycles, we obtain a numerical formula for the depth of the second square-free power. 
\medskip

Let us start by giving tight bounds on the regularity of $k^{th}$ square-free powers of cycles.
\begin{proposition}\label{cycle reg bound}
    Let $C_n$ denote the cycle of length $n$. Then for $1\le k\le \nu(C_n)=\lfloor\frac{n}{2}\rfloor$, we have
    \[
    2k+\left\lfloor \frac{n-2k-1}{3}\right\rfloor\leq \reg(I(C_n)^{[k]})\le 2k+\left\lfloor \frac{n-2k}{3} \right\rfloor.
    \]
\end{proposition}
\begin{proof}
    We first prove the upper bound by induction on $k$. If $k=1$, then the bound follows from \Cref{cycle}. Therefore, we may assume that $k\ge 2$. Let us consider the ideal $I(C_n)^{[k]}+\l x_1x_2,x_1x_n\r$. Since $N_{C_n}(x_1)=\{x_2,x_n\}$, using \Cref{colon with variable}, we get $(I(C_n)^{[k]}+\l x_1x_2,x_1x_n\r):x_1=I(C_n\setminus \{x_1,x_2,x_n\})^{[k]}+\l x_2,x_n\r$. Note that $C_n\setminus \{x_1,x_2,x_n\}\cong P_{n-3}$, and $I(C_n\setminus \{x_1,x_2,x_n\})^{[k]}=0$ if $k\ge \lfloor\frac{n}{2}\rfloor-1$. Therefore, by \Cref{sqf path}, $\reg((I(C_n)^{[k]}+\l x_1x_2,x_1x_n\r):x_1)\le 2k-1+\lfloor\frac{n-2k}{3} \rfloor$. Now, by \Cref{lemcolonedge}, $I(C_n)^{[k]}+\l x_1x_2,x_1x_n,x_1\r=I(C_n)^{[k]}+\l x_1\r=I(C_n\setminus x_1)^{[k]}+\l x_1\r$. Observe that $C_n\setminus x_1\cong P_{n-1}$, and $I(C_n\setminus x_1)^{[k]}=\l 0\r$ if $k=\lfloor\frac{n}{2}\rfloor$. Hence, $\reg(I(C_n)^{[k]}+\l x_1x_2,x_1x_n,x_1\r)\le 2k+\lfloor \frac{n-2k}{3}\rfloor$ (by \Cref{sqf path}). Therefore, using \Cref{lemreg}(iii), we obtain 
    \[
    \reg(I(C_n)^{[k]}+\l x_1x_2,x_1x_n\r)\le 2k+\left\lfloor \frac{n-2k}{3}\right\rfloor.
    \]
    Next, consider the ideal $((I(C_n)^{[k]}+\l x_1x_2\r):x_1x_n)=\l x_2\r+(I(C_n)^{[k]}:x_1x_n)$. Hence, by \Cref{sqf colon}, we have    $(I(C_n)^{[k]}+\l x_1x_2\r):x_1x_n=\l x_2\r+I(C_n\setminus\{x_1,x_2,x_n\})^{[k-1]}$. Again note that $I(C_n\setminus\{x_1,x_2,x_n\})^{[k-1]}=I(P_{n-3})^{[k-1]}=\l 0\r$ if $k=\lfloor\frac{n}{2}\rfloor$. Hence, using \Cref{sqf path} again, we get $\reg((I(C_n)^{[k]}+\l x_1x_2\r):x_1x_n)\le 2k-2+\lfloor \frac{n-2k}{3}\rfloor$. Consequently, by \Cref{lemreg}, 
    \[
    \reg(I(C_n)^{[k]}+\l x_1x_2\r)\le 2k+\left\lfloor \frac{n-2k}{3}\right\rfloor.
    \]
    Note that, by \Cref{sqf colon}, we have $(I(C_n)^{[k]}:x_1x_2)=I(C_{n-2})^{[k-1]}$, where the cycle $C_{n-2}$ is obtained from $C_n$ by removing $x_1,x_2$ and adding the edge $\{x_3,x_n\}$. Thus, by the induction hypothesis, $(\reg(I(C_n)^{[k]}:x_1x_2)\le 2k-2+\lfloor \frac{n-2k}{3}\rfloor$ and finally using \Cref{lemreg}, we get $\reg(I(C_n)^{[k]})\le 2k+\lfloor \frac{n-2k}{3}\rfloor$, as desired. \par 

 To establish the lower bound, we consider the ideal $I(C_n)^{[k]}+\l x_n\r$. Note that by \Cref{sqf colon},  $I(C_n)^{[k]}+\l x_n\r=I(C_{n}\setminus \{x_n\})^{[k]}+\l x_n\r$ and $C_{n}\setminus\{x_n\}\cong P_{n-1}$. Thus, using \Cref{sqf path}, we get $\reg(I(C_n)^{[k]}+\l x_n\r)=2k+\lfloor\frac{n-2k-1}{3}\rfloor$. Hence, by \Cref{lemreg}, $2k+\lfloor \frac{n-2k-1}{3}\rfloor\leq \reg(I(C_n)^{[k]})$.
\end{proof}

Next, our aim is to find an exact formula for the regularity of $I(C_n)^{[k]}$. To establish the formula, we need the regularity bounds from the above proposition as well as some tools from the theory of simplicial complexes. Specifically, we identify square-free powers of edge ideals with facet ideals of certain simplicial complexes. Let us briefly recall some necessary prerequisites. We refer the readers to \cite{Faridi, RHV} for quick references.

A {\it simplicial complex} $\Delta$ on a vertex set $V(\Delta)$ is a non-empty collection of subsets of $V(\Delta)$ such that $\{x_i\}\in\Delta$ for each $x_i\in V(\Delta)$, and if $F'\in\Delta$ and $F\subseteq F'$, then $F\in \Delta$. Each element in $\Delta$ is called a {\it face} of $\Delta$. A maximal face with respect to inclusion is called a {\it facet} of $\Delta$. If $\{F_1,\ldots,F_t\}$ is the collection of all facets of $\Delta$, then we write $\mathcal F(\Delta)=\{F_1,\ldots,F_t\}$, and $\Delta=\l F\mid F\in \mathcal F(\Delta)\r$. If $F$ is a face of $\Delta$, then the {\it dimension of $F$} is the number $|F|-1$, and is denoted by $\mathrm{dim}(F)$. The dimension of the simplicial complex $\Delta$ is defined to be $\mathrm{dim}(\Delta)=\max_{F\in\Delta}\mathrm{dim}(F)$. For any $U\subseteq V(\Delta)$, an {\it induced subcomplex} of $\Delta$ on $U$ is the simplicial complex on the vertex set $U$ and with the set of facets $\{F\in\mathcal{F}(\Delta)\mid F\subseteq U\}$. Moreover, if $\Delta=\l F_1,\ldots,F_t\r$, then $\Delta^c_U$ denote the simplicial complex with the set of facets $\{U\setminus F_1,\ldots,U\setminus F_t\}$.

For a simplicial complex $\Delta$ an {\it orientation} on $\Delta$ is a total order $<$ on $V(\Delta)$. In this case $(\Delta,<)$ is called an {\it oriented simplicial complex}. Let $(\Delta,<)$ be a $t$-dimensional oriented simplicial complex. A face $F=\{w_1,\ldots,w_r\}$ with $w_1<\dots<w_r$ is said to be an {\it oriented face}. In this case, we simply write $F=[w_1,\ldots,w_r]$. Let $\C_i(\Delta)$ denote the $\K$-vector space with basis consisting of all the oriented $i$-dimensional faces of $\Delta$. Then the {\it augmented (oriented) chain complex} of $\Delta$ is the complex
    \[
0\rightarrow\C_t(\Delta)\xrightarrow{\partial_t}\C_{t-1}(\Delta)\xrightarrow{\partial_{t-1}}\cdots\rightarrow\C_1(\Delta)\xrightarrow{\partial_1}\C_0(\Delta)\xrightarrow{\partial_0}\K\rightarrow 0,
    \]
    where $\partial_0([w])=1$ for all $w\in V(\Delta)$, and for each $r\in\{1,\ldots,t\}$, we have the {\it boundary maps} $\partial_r([w_1,\ldots,w_{r+1}])=\sum_{j=1}^{r+1}(-1)^{j+1}[w_1,\ldots,\widehat{w_j},\ldots,w_{r+1}]$, where $w_j$ is not present in the oriented face $[w_1,\ldots,\widehat{w_j},\ldots,w_{r+1}]$.

    Let $I$ be a square-free monomial ideal in $R=\K[x_1,\ldots,x_n]$. Then the simplicial complex $\Delta(I)=\l F\subseteq \{x_1,\ldots,x_n\}\mid \x_F\in\G(I)\r$ is called the {\it facet complex} of $I$, and $I$ is said to be the {\it facet ideal} of $\Delta(I)$. In this case the graded Betti numbers of $I$ are determined by the dimension of the reduced homology groups of the simplicial complex $\Delta(I)$. Recall that, $\widetilde{H}_i(\Delta(I);\K)$ denote the $i^{th}$ reduced simplicial homology group of $\Delta(I)$ over the field $\K$. Then, from \cite[Theorem 2.8]{AlFa} we have
    \begin{align}\label{facet homology formula}
    \beta_{p,q}(I)=\sum_{\substack{\Gamma\subseteq \Delta(I)\\ 
    |V(\Gamma)|=q}}\mathrm{dim}_{\K}\widetilde{H}_{p-1}\left(\Gamma^c_{V(\Gamma)};\K\right),
    \end{align}
    where $\Gamma$ is any induced subcomplex of $\Delta(I)$ such that $|V(\Gamma)|=q$. Using this description of Betti numbers we are able to prove the following theorem.

\begin{theorem}\label{cycle reg theorem}
    Let $C_n$ denote the cycle of length $n$. Then for $1\le k\le \nu(C_n)=\lfloor\frac{n}{2}\rfloor$, we have
    \[
    \reg(I(C_n)^{[k]})=2k+\left\lfloor\frac{n-2k}{3}\right\rfloor.
    \]
\end{theorem}
\begin{proof}
     Observe that if $n-2k$ is not divisible by $3$, then by \Cref{cycle reg bound}, $\reg(I(C_n)^{[k]})=2k+\left\lfloor\frac{n-2k}{3}\right\rfloor$. Moreover, if $n=2k$, then $I(C_n)^{[k]}$ has a linear resolution by \cite[Theorem 5.1]{BHZN}, and thus, we have the required formula. Therefore, it is enough to consider the case when $n-2k=3j$ for some $j\ge 1$. Now let $\Delta(I(C_n)^{[k]})$ denote the facet complex corresponding to the ideal $I(C_n)^{[k]}$. Our aim now is to show that $\beta_{2j,n}(I(C_n)^{[k]})\neq 0$, where $n=2k+3j$ for some $j\ge 1$. Using \Cref{facet homology formula} we have 
    \begin{align*}
    \beta_{2j,n}(I(C_n)^{[k]})=\mathrm{dim}_{\K}\widetilde{H}_{2j-1}\left(\Delta(I(C_n)^{[k]})^c;\K\right).
    \end{align*}
     Thus it is enough to show that $\widetilde{H}_{2j-1}\left(\Delta(I(C_n)^{[k]})^c;\K\right)\neq 0$. Let $X= \{2k+3l,2k+3l+1\mid 0\le l\le j-1\}\cup\{n\}\subseteq\{1,\ldots,n\}$, where $|X|=2j+1$. Consider the element
    \[
    [\alpha]=\sum_{s=1}^{2j+1}(-1)^{s+1}[w_1,\ldots,\widehat{w_s},\ldots,w_{2j+1}],
    \]
    where $w_r\in X$ for all $r=1,\ldots,2j+1$, and $w_r<w_{r+1}$ for all $r=1,\ldots,2j$. We proceed to show that $[\alpha]\in \mathrm{ker}(\partial_{2j-1})$, but $[\alpha]\notin \mathrm{Im}(\partial_{2j})$, and consequently, $\beta_{2j,n}(I(C_n)^{[k]})\neq 0$. Our arguments closely follow the proof of \cite[Theorem 4.7]{KaNaQu}.

    For $1\le s\le 2j+1$, let $[\lambda_s]$ denote the oriented face $[w_1,\ldots,\widehat{w_s},\ldots,w_{2j+1}]$ of $\Delta(I(C_n)^{[k]})^c$. Then $[\alpha]=\sum_{s=1}^{2j+1}(-1)^{s+1}[\lambda_s]$. Our first claim is the following.

    \noindent
    \textbf{Claim 1}: $[\alpha]\in\mathrm{ker}(\partial_{2j-1})$.

    \noindent
    \textit{Proof of Claim 1}: First we need to prove that $[\alpha]\in \C_{2j-1}(\Delta(I(C_n)^{[k]})^c)$. For this, it is enough to show that $[\lambda_s]\in \C_{2j-1}(\Delta(I(C_n)^{[k]})^c)$ for each $s\in\{1,\ldots,2j+1\}$. If $s=1$, then $[\lambda_1]\in \C_{2j-1}(\Delta(I(C_n)^{[k]})^c)$ since $\{1,\ldots,2k\}$ is a facet of $\Delta(I(C_n)^{[k]})$. Similarly, if $s=2j+1$, then $[\lambda_{2j+1}]\in \C_{2j-1}(\Delta(I(C_n)^{[k]})^c)$ since $\{1,\ldots,2k-2\}\cup\{n,n-1\}$ is a facet of $\Delta(I(C_n)^{[k]})$. Now if $1<s<2j+1$, then $[\lambda_{s}]\in \C_{2j-1}(\Delta(I(C_n)^{[k]})^c)$ since $\{1,\ldots,2k-2\}\cup\{2k+3l-1,2k+3l\}$ is a facet of $\Delta(I(C_n)^{[k]})$ in case $w_s=2k+3l$, and $\{1,\ldots,2k-2\}\cup\{2k+3l+1,2k+3l+2\}$ is a facet of $\Delta(I(C_n)^{[k]})$ in case $w_s=2k+3l+1$, where $0\le l\le j-1$. Next, observe that
    \begin{align*}
        \partial_{2j-1}([\alpha])
        &=\sum_{s=1}^{2j+1}(-1)^{s+1}\partial_{2j-1}([w_1,\ldots,\widehat{w_s},\ldots,w_{2j+1}])\\
        &=\sum_{s=1}^{2j+1}(-1)^{s+1}\left(\sum_{r=s+1}^{2j+1}(-1)^{r}[w_1,\ldots,\widehat{w_s},\ldots,\widehat{w_r}\ldots,w_{2j+1}]+\right.\\
&\quad\quad\quad\quad\quad\quad\quad\quad\quad\quad\quad\quad\quad\quad\quad\left. \sum_{r=1}^{s-1}(-1)^{r+1}[w_1,\ldots,\widehat{w_r},\ldots,\widehat{w_s}\ldots,w_{2j+1}]\right)\\
        &=0.
    \end{align*}
    This completes the proof of Claim 1.

    To prove $[\alpha]\notin \mathrm{Im}(\partial_{2j})$, first notice that $[X]=[w_1,\ldots,w_{2j+1}]\notin \C_{2j}(\Delta(I(C_n)^{[k]})^c)$ since $\{1,\ldots,n\}\setminus X$ does not contain any $k$-matching of $C_n$. Now let $\B_{2j}$ (respectively, $\B_{2j-1}$) denote the basis of $\C_{2j}(\Delta(I(C_n)^{[k]})^c)$ (respectively, $\C_{2j-1}(\Delta(I(C_n)^{[k]})^c)$) consisting of all the $2j$-dimensional (respectively, $(2j-1)$-dimensional) oriented faces of $\Delta(I(C_n)^{[k]})^c$. Let $A$ be the matrix corresponding to the linear transformation $\partial_{2j}$, and $\b_{[\alpha]}$ denote the coordinate vector (written as a column vector) corresponding to the element $[\alpha]\in\C_{2j-1}(\Delta(I(C_n)^{[k]})^c)$. Then it is enough to show that the system of linear equations $A\x=\b_{[\alpha]}$ has no solution in $\K^n$. For this, it is enough to show that $\mathrm{rank}(A)$ is strictly less than the rank of the augmented matrix $(A\mid \b_{[\alpha]})$. We show this by performing some elementary row operations on $(A\mid\b_{[\alpha]})$ described below.

    Recall that the rows and columns of the matrix $A$ are indexed by the basis elements $\B_{2j-1}$ and $\B_{2j}$, respectively. Let $R$ denote the row in $(A\mid\b_{[\alpha]})$ corresponding to $[\lambda_1]\in \B_{2j-1}$. Next, for $1\le t\le j$, and for all distinct even numbers $i_1,\ldots,i_t\in\{2,\ldots,2j\}$, let $R(i_1,\ldots,i_t)$ denote the row corresponding to the basis element $[(\lambda_1\setminus\{w_{i_1},\ldots,w_{i_t}\})\cup\{w_{i_1}+1,\ldots,w_{i_t}+1\}]\in\B_{2j-1}$. Moreover, given such a $t$ and for $1\le s\le t$, let $C(i_1,\ldots,\overline{i_s},\ldots,i_t)$ denote the column corresponding to the basis element $[(\lambda_1\setminus\{w_{i_1},\ldots,w_{i_t}\})\cup\{w_{i_s},w_{i_1}+1,\ldots,w_{i_t}+1\}]\in\B_{2j}$. For the row $R$ in $(A\mid\b_{[\alpha]})$ one can observe the following.
    \begin{enumerate}
        \item[(i)] The entry corresponding to the column $C(\overline{i})$ is $-1$, where $i\in\{2,\ldots, 2j\}$ is an even number.

        \item[(ii)] The entry corresponding to the column $\b_{[\alpha]}$ is $1$.

        \item[(iii)] The entries corresponding to all other columns are zero. This happens because if $w\in \{1,\ldots,n\}\setminus (X\setminus\{2k\})$, then $\{1,\ldots,n\}\setminus(\lambda_1\cup\{w\})$ does not contain any $k$-matching of the cycle $C_n$.
    \end{enumerate}
 Based on these observations, we perform the elementary row operations in $j$-many steps on $R$ to make all the entries in $R$ corresponding to the columns indexed by all the elements of $\B_{2j}$ to be zero, whereas the entry corresponding to the column $\b_{[\alpha]}$ in $R$ to be $1$.

\noindent
\textbf{Step-$\mathbf{1}$}: Perform the row operation $R+\sum_{\substack{i\in\{2,\ldots,2j\}\\
i\text{ even }}} R(i)$ on the row $R$, and let $R_1$ denote the row $R$ after this operation.

\noindent
For $2\le t\le j$, go through the following steps one by one.

\noindent
\textbf{Step-$\mathbf{t}$}: Perform the operation $R_{t-1}+\sum_{\substack{i_1,i_2,\ldots,i_t\in\{2,\ldots,2j\}\\
i_1,i_2,\ldots,i_t\text{ distinct even numbers}}}R(i_1,\ldots,i_t)$ on the row $R_{t-1}$, and let $R_t$ denote the row $R_{t-1}$ after this operation.

\noindent
\textbf{Claim 2}: In $R_1$ the entry corresponding to the column $\b_{[\alpha]}$ is $1$. Moreover, in $R_1$ the entry corresponding to the column $C(i,\overline{i'})$ is $-1$, where $i,i'\in\{2,\ldots,2j\}$ are distinct even numbers; entries corresponding to the all other columns are zero.

\noindent
\textit{Proof of Claim 2}: Observe that the entry in $R(i)$ corresponding to the column $\b_{[\alpha]}$ is $0$ since $w_{i}+1$ appears in the basis element corresponding to the row $R(i)$. Hence, in $R_1$, the entry corresponding to the column $\b_{[\alpha]}$ is $1$. Note that in $R(i)$, the only possible non-zero entries are the entries $1$ and $-1$ corresponding to the columns $C(\overline{i})$ and $C(i,\overline{i'})$, respectively, where $i,i'\in\{2,\ldots,2j\}$ are distinct even numbers. Now for any even $i_1\in\{2,\ldots,2j\}$ with $i_1\neq i$, the entry corresponding to the column $C(\overline{i})$ in $R(i_1)$ is zero since $w_{i_1}+1$ appears in the basis element corresponding to $R(i_1)$, whereas $w_{i_1}+1$ does not appear in the basis element corresponding to the column $C(\overline{i})$. Thus, the entry in $R_1$ corresponding to the column $C(\overline{i})$ is $0$. Again, observe that the entry in $R(i_1)$ corresponding to the column $C(i,\overline{i'})$ is $0$ since $w_i$ appears in the basis element corresponding to the row $R(i_1)$, whereas $w_i$ does not appear in the basis element corresponding to the column $C(i,\overline{i'})$. Thus, the entry corresponding to the column $C(i,\overline{i'})$ in $R_1$ is $-1$. Since $R_1=R+\sum_{\substack{i\in\{2,\ldots,2j\}\\
i\text{ even }}} R(i)$, we see that all other entries in $R_1$ are zero.

\noindent
\textbf{Claim 3}: For each $2\le t\le j-1$, the entry corresponding to the column $\b_{[\alpha]}$ in $R_t$ is $1$. Moreover, in $R_t$, the entry corresponding to the column $C(i_1,i_2,\ldots,i_t,\overline{i})$ is $-1$, where $i_1,\ldots,i_t,i\in\{2,\ldots,2j\}$ are distinct even numbers; entries corresponding to all other columns in $R_t$ are zero.

\noindent
\textit{Proof of Claim 3}: From Step $t-1$ we have that the entry in $R_{t-1}$ corresponding to the column $\b_{[\alpha]}$ is $1$. Now for distinct even numbers $i_1,\ldots,i_t\in\{2,\ldots,2j\}$, the entry in $R(i_1,\ldots,i_t)$ corresponding to the column $\b_{[\alpha]}$ is $0$ since $w_{i_1}+1$ appears in the basis element corresponding to the row $R(i_1,\ldots,i_t)$. Consequently, the entry corresponding to the column $\b_{[\alpha]}$ in $R_t$ is $1$. Next, we see from Step $t-1$ that the only possible non-zero entries in $R_{t-1}$ are the entries $-1$ corresponding to the columns $C(i_1,\ldots,i_{t-1},\overline{i})$, where $i_1,\ldots,i_{t-1},i\in\{2,\ldots,2j\}$ are distinct even numbers. Note that in $R(i_1,\ldots,i_t)$, the only possible non-zero entries are $1$ and $-1$ corresponding to the columns $C(i_1,\ldots,\overline{i_s},\ldots,i_t)$ and $C(i_1,\ldots,i_t,\overline{i})$, respectively, where $i,i_1,\ldots,i_s,\ldots,i_t\in\{2,\ldots,2j\}$ are distinct even numbers. We proceed to show that if $i_1',\ldots,i_t'\in\{2,\ldots,2j\}$ are distinct even numbers such that $\{i_1,\ldots,i_t\}\neq\{i_1',\ldots,i_t'\}$, then the entry in $R(i_1',\ldots,i_t')$ corresponding to the column $C(i_1,\ldots,\overline{i_s},\ldots,i_t)$ is $0$. Indeed, $i_m'\notin \{i_1,\ldots,i_t\}$ for some $m\in\{1,\ldots,t\}$ so that $w_{i_m'}+1$ appears in the basis element corresponding to the row $R(i_1'\ldots,i_t')$, whereas $w_{i_m'}+1$ does not appear in the basis element corresponding to the column $C(i_1,\ldots,\overline{i_s},\ldots,i_t)$. Consequently, in $R_t$, the entry corresponding to the column $C(i_1,\ldots,\overline{i_s},\ldots,i_t)$ is $0$. Next, observe that the entry in $R(i_1',\ldots,i_t')$ corresponding to the column $C(i_1,\ldots,i_t,\overline{i})$ is $0$ since $i_l\notin \{i_1',\ldots,i_t'\}$ for some $l\in \{1,\ldots,t\}$, which in turn, implies that $w_{i_l}$ appears in the basis element corresponding to the row $ R(i_1',\ldots,i_t')$, whereas $w_{i_l}$ does not appear in the basis element corresponding to the column $C(i_1,\ldots,i_t,\overline{i})$. Thus, the entry in $R_t$ corresponding to the column $C(i_1,\ldots,i_t,\overline{i})$ is $-1$. Since $R_t=R_{t-1}+\sum_{\substack{i_1,i_2,\ldots,i_t\in\{2,\ldots,2j\}\\
i_1,i_2,\ldots,i_t\text{ distinct even numbers}}}R(i_1,\ldots,i_t)$, all other entries in $R_t$ are zero.

\noindent
\textbf{Claim 4}: The entry corresponding to the column $\b_{[\alpha]}$ in $R_j$ is $1$. All other entries in $R_j$ are zero.

\noindent
\textit{Proof of Claim 4}: From Step $j-1$ we see that the entry in $R_{j-1}$ corresponding to the column $\b_{[\alpha]}$ is $1$. Now, observe that $R_j=R_{j-1}+R(2,4,\ldots,2j)$, and since $w_2+1=2k+2$ appears in the basis element corresponding to $R(2,4,\ldots,2j)$, we see that the entry in $R_j$ corresponding to the column $\b_{[\alpha]}$ is $1$. Again, from Step $j-1$ we observe that the only possible non-zero entries in $R_{j-1}$ are the entries $-1$ corresponding to the columns $C(i_1,\ldots,i_{j-1},\overline{i})$, where $i_1,\ldots,i_{j-1},i\in\{2,\ldots,2j\}$ are distinct even numbers. In this case, it is easy to see that $\{2,4,\ldots,2j\}=\{i_1,\ldots,i_{j-1},i\}$. Note that the only possible non-zero entries in $R(2,4,\ldots,2j)$ are $1$ corresponding to the columns $C(i_1,\ldots,i_{j-1},\overline{i})$, where $\{i_1,\ldots,i_{j-1},i\}=\{2,4,\ldots,2j\}$. Thus, we can conclude that entries corresponding to all the columns except $\b_{[\alpha]}$ in $R_j$ are zero.

From Claim 4, we can deduce that $\mathrm{rank}(A)<\mathrm{rank}(A\mid \b_{[\alpha]})$. Consequently, $\beta_{2j,n}(I(C_n)^{[k]})\neq 0$, and thus, $\reg(I(C_n)^{[k]})=2k+\lfloor\frac{n-2k}{3}\rfloor$, as desired.
\end{proof}

We illustrate the steps involved in \Cref{cycle reg theorem} in the following example.
\begin{example}
    Let us consider the ideal $I(C_{13})^{[2]}$ and proceed as in \Cref{cycle reg theorem} to show that $\beta_{6,13}(I(C_{13})^{[2]})\neq 0$. In this case $X=\{4,5,7,8,10,11,13\}$. Then
    \begin{align*}
    [\alpha]&=[5,7,8,10,11,13]-[4,7,8,10,11,13]+[4,5,8,10,11,13]-[4,5,7,10,11,13]+\\
    &\quad\quad\quad\quad\quad\quad\quad\quad\quad\quad\quad\quad\quad\quad\quad\quad[4,5,7,8,11,13]-[4,5,7,8,10,13]+[4,5,7,8,10,11].    
    \end{align*}
    Observe the following:
    \begin{enumerate}
        \item[$\bullet$] $[5,7,8,10,11,13]\in \Delta(I(C_{13})^{[2]})^c$ since $\{1,2,3,4\}$ is a facet in $\Delta(I(C_{13})^{[2]})$;

        \item[$\bullet$] $[4,7,8,10,11,13]\in \Delta(I(C_{13})^{[2]})^c$ since $\{1,2\}\cup\{5,6\}$ is a facet in $\Delta(I(C_{13})^{[2]})$;

        \item[$\bullet$] $[4,5,8,10,11,13]\in \Delta(I(C_{13})^{[2]})^c$ since $\{1,2\}\cup\{6,7\}$ is a facet in $\Delta(I(C_{13})^{[2]})$;

        \item[$\bullet$] $[4,5,7,10,11,13]\in \Delta(I(C_{13})^{[2]})^c$ since $\{1,2\}\cup\{8,9\}$ is a facet in $\Delta(I(C_{13})^{[2]})$;

        \item[$\bullet$] $[4,5,7,8,11,13]\in \Delta(I(C_{13})^{[2]})^c$ since $\{1,2\}\cup\{9,10\}$ is a facet in $\Delta(I(C_{13})^{[2]})$;

        \item[$\bullet$] $[4,5,7,8,10,13]\in \Delta(I(C_{13})^{[2]})^c$ since $\{1,2\}\cup\{11,12\}$ is a facet in $\Delta(I(C_{13})^{[2]})$;

        \item[$\bullet$] $[4,5,7,8,10,11]\in \Delta(I(C_{13})^{[2]})^c$ since $\{1,2\}\cup\{12,13\}$ is a facet in $\Delta(I(C_{13})^{[2]})$.
    \end{enumerate}
    One can see that $\partial_5([\alpha])=0$. Moreover, $[4,5,7,8,10,11,13]\notin \Delta(I(C_{13})^{[2]})^c$ since there is no $2$-matching in the induced subgraph of $C_{13}$ corresponding to the set of vertices $\{1,2,3,6,9,12\}$. Next, in order to show that $[\alpha]\notin\mathrm{Im}(\partial_6)$, we consider the row $R$ in the augmented matrix $(A\mid \b_{\alpha})$ corresponding to the basis element $[5,7,8,10,11,13]\in\B_5$. The non-zero entries in the row $R$ are depicted in \Cref{Example row R}. Here, for example, the column $C(\overline{8})$ indicates the column corresponding to the basis element $[5,7,8,9,10,11,13]\in \B_6$.
    
    \begin{table}[h!]
	\centering
	\footnotesize
	\scalebox{1.2}{
		\begin{tabular}{c|c|c|c|c}
			& $C(\overline{5})$ & $C(\overline{8})$ & $C(\overline{11})$  & $\mathbf{b}_\alpha$     \\ \hline
			$R$  & $-1$           & $-1$           & $-1$            & $1$                     
		\end{tabular}
	}
	\caption{The non-zero entries of $R$ in $(A \vert \mathbf{b}_\gamma)$.}
	\label{Example row R}	
\end{table}

In \Cref{Example Step 1} we present the relevant rows used in Step $1$. Here, for example, $R(5)$ and $C(5,\overline{11})$ denote the row and column corresponding to the basis elements $[6,7,8,10,11,13]\in\B_5$ and $[6,7,8,10,11,12,13]\in\B_6$, respectively. In each row, $0$ appears in the remaining columns which are not drawn here due to the matrix's large size.
\begin{table}[h!]
	\centering
	\footnotesize
	\scalebox{1.1}{
		\begin{tabular}{c|c|c|c|c|c|c|c|c|c|c}
			& $C(\overline{5})$ & $C(\overline{8})$ & $C(\overline{11})$ & $C(5,\overline{8})$ & $C(\overline{5},8)$ & $C(5,\overline{11})$ & $C(\overline{5},11)$ & $C(8,\overline{11})$ & $C(\overline{8},11)$ & $\mathbf{b}_\alpha$     \\ \hline
			$R(5)$  & $1$ & $0$ & $0$ & $-1$ & $0$ & $-1$ & $0$ & $0$ & $0$ & $0$\\
   \hline $R(8)$ & $0$ & $1$ & $0$ & $0$ & $-1$ & $0$ & $0$ & $-1$ & $0$ & $0$\\
   \hline $R(11)$ & $0$ & $0$ & $1$ & $0$ & $0$ & $0$ & $-1$ & $0$ & $-1$ & $0$
		\end{tabular}
	}
	\caption{Rows used for the elementary operation in step $1$.}
	\label{Example Step 1}	
\end{table}

In \Cref{Example Step 2} we present the relevant rows used in Step $2$. Here, for example, $R(8,11)$ and $C(5,8,\overline{11})$ denote the row and column corresponding to the basis elements $[5,7,9,10,12,13]\in\B_5$ and $[6,7,9,10,11,12,13]\in\B_6$, respectively. In each row, $0$ appears in the remaining columns which are again not drawn here due to the matrix's large size.
\begin{table}[h!]
	\centering
	\footnotesize
	\scalebox{0.95}{
		\begin{tabular}{c|c|c|c|c|c|c|c|c|c|c}
			& $C(\overline{5},8)$ & $C(5,\overline{8})$ & $C(\overline{5},11)$ & $C(5,\overline{11})$ &   $C(\overline{8},11)$ & $C(8,\overline{11})$ & $C(\overline{5},8,11)$ & $C(5,\overline{8},11)$ & $C(5,8,\overline{11})$ & $\mathbf{b}_\alpha$     \\ \hline
			$R(5,8)$  & $1$ & $1$ & $0$ & $0$ & $0$ & $0$ & $0$ & $0$ & $-1$ & $0$\\
   \hline $R(5,11)$ & $0$ & $0$ & $1$ & $1$ & $0$ & $0$ & $0$ & $-1$ & $0$ & $0$\\
   \hline $R(8,11)$ & $0$ & $0$ & $0$ & $0$ & $1$ & $1$ & $-1$ & $0$ & $0$ & $0$
		\end{tabular}
	}
	\caption{Rows used for the elementary operation in step $2$.}
	\label{Example Step 2}	
\end{table}

\begin{table}[h!]
	\centering
	\footnotesize
	\scalebox{1.2}{
		\begin{tabular}{c|c|c|c|c}
			& $C(\overline{5},8,11)$ & $C(5,\overline{8},11)$ & $C(5,8,\overline{11})$ & $\mathbf{b}_\alpha$     \\ \hline
			$R(5,8,11)$  & $1$ & $1$ & $1$ & $0$
		\end{tabular}
	}
	\caption{Row used for the elementary operation in the final step.}
	\label{Example last step}	
\end{table}
In \Cref{Example last step}, we depict the row $R(5,8,11)$ and its non-zero columns along with $\b_{\alpha}$. Here $R(5,8,11)$ denotes the row corresponding to the basis element $[6,7,9,10,12,13]\in\B_5$. It is easy to see that after this step all entries in the row $R$ become $0$, except the entry $1$ corresponding to the column $\b_{\alpha}$. Thus $\mathrm{rank}(A)<\mathrm{rank}(A\mid\b_{\alpha})$.

\end{example}

\begin{corollary}\label{linear even cycle}
    Let $C_n$ be a cycle of length $n$. Then $I(C_n)^{[k]}$ has a linear resolution if and only if $k=\nu(C_n)$ in case of odd $n$, and $k\in\{\nu(C_n),\nu(C_n)-1\}$ in case of even $n$.
\end{corollary}
\begin{proof}
    Follows from \Cref{cycle reg theorem}.
\end{proof}

We now proceed to investigate the depth of square-free powers of edge ideals of cycles. First, we need the following lemmas, which deal with the depth of square-free powers of edge ideals of paths and their induced subgraphs.

\begin{lemma}\label{path depth}\cite[cf. Theorem 2.10]{CFL1}
    Let $P_n$ denote the path graph on $n$ vertices. Then for each $1\le k\le \nu(P_n)=\lfloor\frac{n}{2}\rfloor$,
    \[
    \d(R/I(P_n)^{[k]})=\begin{cases}
        \lceil\frac{n}{3}\rceil+k-1&\text{ for }1\le k\le \lceil\frac{n}{3}\rceil,\\
        2k-1&\text{ otherwise.}
    \end{cases}
    \]
\end{lemma}

\begin{remark}\label{path depth remark}
    From \Cref{path depth} it is easy to observe that, for $1\le k\le \lfloor\frac{n}{2}\rfloor$, $\d(R/I(P_n)^{[k]})\ge \lceil\frac{n}{3}\rceil+k-1$.
\end{remark}

\begin{lemma}\label{path comma proposition}
    Let $P_n$ denote the path graph on $n$ vertices and $1\le k\le\lfloor\frac{n}{2}\rfloor$ be an integer. Let $H$ denote an induced path of $P_n$ containing a simplicial vertex of $P_n$ such that $H$ contains at most $n-2k+1$ vertices. Then $\mathrm{depth}(R/I(P_n)^{[k]}+I(H))\ge \lceil \frac{n}{3}\rceil+k-1$.
\end{lemma}
\begin{proof}
    Let $V(P_n)=\{x_1,\ldots,x_n\}$ with $E(P_n)=\{\{x_ix_{i+1}\}\mid 1\le i\le n-1\}$, and let $1\le m\le n-2k+2$ be an integer. If $k=1$, then the formula follows from the above remark. Hence we may assume that $k\ge 2$. Without loss of generality, let $x_1$ be the simplicial vertex of $P_n$ contained in $H$. We show by induction on $m$ that if the number of vertices in $H$ is $n-2k+2-m$, then $\mathrm{depth}(R/I(P_n)^{[k]}+I(H))\ge \lceil \frac{n}{3}\rceil+k-1$. First, note that if $m=n-2k+2$ or $m=n-2k+1$, then $I(P_n)^{[k]}+I(H)=I(P_n)^{[k]}$, and thus, by \Cref{path depth} and \Cref{path depth remark}, $\d(R/I(P_n)^{[k]}+I(H))\ge \lceil \frac{n}{3}\rceil+k-1$. Therefore, we may assume that $m\le n-2k+3$. Now consider the case $m=1$. In this case, $E(H)=\{x_1x_2,\ldots,x_{n-2k}x_{n-2k+1}\}$. Then it is easy to see that 
    \[
    I(P_n)^{[k]}+I(H)=\left\l \prod_{j=n-2k+1}^nx_j,x_1x_2,x_2x_3,\ldots,x_{n-2k}x_{n-2k+1}\right\r.
    \]
    Notice that if $J_1$ denotes the ideal $((I(P_n)^{[k]}+I(H)):x_{n-2k})$, then $J_1=\l x_{n-2k-1},x_{n-2k+1}\r+\l x_1x_2,\ldots,x_{n-2k-3}x_{n-2k-2}\r$. Thus, by \Cref{lemdepthextravar} \& \ref{path depth}, $\d(R/J_1)= \lceil \frac{n-2k-2}{3} \rceil+2k\ge \lceil\frac{n}{3}\rceil+k-1$. Now suppose $J_2$ denotes the ideal $I(P_n)^{[k]}+I(H)+\l x_{n-2k}\r$. Then, 
    \[
    J_2=\left\l x_1x_2,\ldots,x_{n-2k-2}x_{n-2k-1},\prod_{j=n-2k+1}^nx_j,x_{n-2k}\right\r.
    \]
    Hence again by \Cref{depth sum}, \ref{lemdepthextravar} \& \ref{path depth}, $\d(R/J_2)=\lceil \frac{n-2k-1}{3}\rceil+2k-1\ge \lceil\frac{n}{3}\rceil+k-1$. Therefore, by \Cref{lemdepthequal}, $\d(R/I(P_n)^{[k]}+I(H))\ge \lceil\frac{n}{3}\rceil+k-1$. 

    Now suppose $H$ is an induced path on $n-2k+2-m$ number of vertices containing the vertex $x_1$, where $m\ge 2$. In that case, if $J_3$ denotes the ideal $I(P_n)^{[k]}+I(H)+\l x_{n-2k+2-m}x_{n-2k+3-m}\r$, then by the induction hypothesis $\d(R/J_3)\ge \lceil\frac{n}{3}\rceil+k-1$. Now consider the ideal $J_4=(I(P_n)^{[k]}+I(H)): x_{n-2k+2-m}x_{n-2k+3-m}$. By \Cref{lemcolonedge}, $J_4=\l x_1x_2,\ldots,x_{n-2k-1-m}x_{n-2k-m}\r+\l x_{n-2k+1-m}\r+I(P_{2k-3+m})^{[k-1]}$, where $P_{2k-3+m}$ is the induced path of $P_n$ on the vertex set $\{x_{n-2k+4-m},\ldots,x_n\}$. Therefore, by \Cref{depth sum}, \ref{lemdepthextravar} and \Cref{path depth remark}, $\d(R/J_4)\ge \lceil\frac{n-2k-m}{3}\rceil+\lceil\frac{2k-3+m}{3}\rceil+(k-1)-1+2\ge \lceil\frac{n}{3}\rceil+k-1$. Finally using \Cref{lemdepthequal}, we obtain $\d(R/I(P_n)^{[k]}+I(H))\ge\lceil\frac{n}{3}\rceil+k-1$.
\end{proof}

The following lemma is crucial in providing the lower bound for the depth of square-free powers of edge ideals of cycles in \Cref{cycle depth lower bound}.

\begin{lemma}\label{cycle depth path comma lower bound}
    Let $C_n$ denote the cycle on $n$ vertices, where $n\ge 4$, and let $2\le k\le \lfloor\frac{n}{2}\rfloor$ be an integer. Let $H$ denote an induced path of $C_n$ containing at most $n-2k+2$ vertices and also containing at least one edge. Then $\mathrm{depth}(R/I(C_n)^{[k]}+I(H))\ge \lceil \frac{n}{3}\rceil+k-1$.
\end{lemma}
\begin{proof}
    Let $V(C_n)=\{x_1,\ldots,x_n\}$ with $E(C_n)=\{\{x_ix_{i+1}\},\{x_1,x_n\}\mid 1\le i\le n-1\}$, and let $ 1\le m\le  n-2k+1$ be an integer. We show by induction on $m$ that if the number of vertices in $H$ is $n-2k+3-m$, then $\mathrm{depth}(R/I(C_n)^{[k]}+I(H))\ge \lceil \frac{n}{3}\rceil+k-1$. First, consider the case $m=1$. Without loss of generality, we can assume that $E(H)=\{x_1x_2,x_2x_3,\ldots,x_{n-2k+1}x_{n-2k+2}\}$. Then it is easy to see that 
    \[
    I(C_n)^{[k]}+I(H)=\left\l x_1\left(\prod_{j=n-2k+2}^nx_j\right), x_1x_2,x_2x_3,\ldots,x_{n-2k+1}x_{n-2k+2}\right\r.
    \]
        Notice that if $J_1$ denotes the ideal $((I(C_n)^{[k]}+I(H)):x_{n-2k+1})$, then $J_1=\l x_{n-2k},x_{n-2k+2}\r+\l x_1x_2,\ldots,x_{n-2k-2}x_{n-2k-1}\r$. Thus, by \Cref{lemdepthextravar} \& \ref{path depth}, $\d(R/J_1)= \lceil \frac{n-2k-1}{3} \rceil+2k-1\ge \lceil\frac{n}{3}\rceil+k-1$. Now, let us denote the ideal $I(C_n)^{[k]}+I(H)+\l x_{n-2k+1}\r$ by $J_2$. Then 
        \[
        J_2=\left\l x_1\left(\prod_{j=n-2k+2}^nx_j\right), x_1x_2,\ldots,x_{n-2k-1}x_{n-2k},x_{n-2k+1}\right\r.
        \]
        Next consider the ideal $(J_2:x_1)=\l x_2,x_{n-2k+1},\left(\prod_{j=n-2k+2}^nx_j\right), x_3x_4,x_4x_5,\ldots,x_{n-2k-1}x_{n-2k}\r$. By \Cref{depth sum}, \ref{lemdepthextravar} \& \ref{path depth}, $\d(R/(J_2:x_1))=\lceil \frac{n-2k-2}{3} \rceil+2k-2+1\ge \lceil\frac{n}{3}\rceil+k-1$. Moreover, we have $J_2+\l x_1\r=\l x_1,x_{n-2k+1},x_2x_3,x_3x_4,\ldots,x_{n-2k-1}x_{n-2k}\r$, and hence, using \Cref{depth sum}, \ref{lemdepthextravar} \& \ref{path depth} again, we obtain $\d(R/J_2+\l x_1\r)\ge \lceil \frac{n-2k-1}{3} \rceil+2k-1\ge\lceil\frac{n}{3}\rceil+k-1$. Now using \Cref{lemdepthequal}, we get $\d(R/J_2)\ge \lceil\frac{n}{3}\rceil+k-1$. Since $\d(R/J_1)\ge\lceil\frac{n}{3}\rceil+k-1$, apply \Cref{lemdepthequal} one more time to get $\d(R/I(C_n)^{[k]}+I(H))\ge \lceil\frac{n}{3}\rceil+k-1$.
        \par 
        
         Suppose $H$ is an induced path on $n-2k+3-m$ number of vertices, where $m\ge 2$. Without loss of generality, we can assume that $E(H)=\{x_1x_2,\ldots,x_{n-2k+2-m}x_{n-2k+3-m}\}$. In this case, if $J_3$ denotes the ideal $I(C_n)^{[k]}+I(H)+\l x_{n-2k+3-m}x_{n-2k+4-m}\r$, then by the induction hypothesis, $\d(R/J_3)\ge \lceil\frac{n}{3}\rceil+k-1$. Now consider the ideal $J_4=((I(C_n)^{[k]}+I(H)): x_{n-2k+3-m}x_{n-2k+4-m})$. Then by \Cref{lemcolonedge}, 
         \[
         J_4=\l x_{n-2k+2-m},x_1x_2,x_3x_4, \ldots,x_{n-2k-m}x_{n-2k+1-m}\r+I(P_{n-3})^{[k-1]},
         \]
         where $P_{n-3}$ is the induced path of $C_n$ on the vertex set $V(C_n)\setminus\{x_{n-2k+2-m},x_{n-2k+3-m},x_{n-2k+4-m}\}$. Note that the ideal $J_4$ can be written as $I(P_{n-3})^{[k-1]}+I(H)+\l x_{n-2k+2-m}\r$, where $H$ is the induced path of the corresponding $P_{n-3}$ on the vertex set $\{x_1,x_2,\ldots,x_{n-2k+1-m}\}$. Observe that $|V(H)|\le (n-3)-2(k-1)+1$ since $m\ge 2$. Therefore, by \Cref{lemdepthextravar} \& \ref{path comma proposition}, $\d(R/J_4)\ge \lceil\frac{n}{3}\rceil+k-1$. Finally using \Cref{lemdepthequal}, we obtain $\d(R/I(P_n)^{[k]}+I(H))\ge\lceil\frac{n}{3}\rceil+k-1$.
         %Now suppose $r\ge n-2k+2$. In that case $I(C_n)^{[k]}+I(H)=I(C_n)^{[k]}$. We show by induction on $n$ that $\d()$ Now by \Cref{lemcolonedge}, $I(C_n)^{[k]}:x_1x_2=I(C_{n-2})^{[k-1]}$, where $E(C_{n-2})=\{\{x_ix_{i+1}\},\{x_3,x_n\}\mid 3\le i\le n-1\}$. Hence using \Cref{lemdepthextravar}
\end{proof}

The depth of the edge ideal of a cycle is well-known and can be easily derived from the thesis of Jacques.

\begin{proposition}\label{cycle edge ideal depth}\cite[cf. Corollary 7.6.30]{JacquesThesis}
    Let $C_n$ denote the cycle on $n$ vertices. Then $$\d(R/I(C_n))=\left \lceil\frac{n-1}{3}\right\rceil.$$
\end{proposition}

Now we are ready to prove a sharp lower bound for the depth of square-free powers of edge ideals of cycles. Note that $I(C_3)^{[k]}=\l 0\r$ if $k>1$. Therefore, it is enough to consider $n\geq 4$ to investigate the square-free powers of $I(C_n)$.

\begin{theorem}\label{cycle depth lower bound}
    Let $C_n$ denote the cycle on $n$ vertices with $n\ge 4$, and let $2\le k\le \lfloor\frac{n}{2}\rfloor$ be an integer. Then $\d(R/I(C_n)^{[k]})\ge \lceil\frac{n}{3}\rceil+k-1$.
\end{theorem}
\begin{proof}
    The proof is by induction on $k$. First, consider the case $k=2$. In this case, by \Cref{lemcolonedge}, $(I(C_n)^{[2]}:x_1x_2)=I(C_{n-2})$, where $E(C_{n-2})=\{\{x_ix_{i+1}\},\{x_3,x_n\}\mid 3\le i\le n-1\}$. Thus, using \Cref{lemdepthextravar} and \Cref{cycle edge ideal depth}, we obtain $\d(R/(I(C_n)^{[2]}:x_1x_2))=2+\lceil\frac{n-3}{3}\rceil=\lceil\frac{n}{3}\rceil+1$. Moreover, by \Cref{cycle depth path comma lower bound}, $\d(R/I(C_n)^{[2]}+\l x_1x_2\r)\ge \lceil\frac{n}{3}\rceil+1$ and consequently, using \Cref{lemdepthequal}, we get $\d(R/I(C_n)^{[2]})\ge \lceil\frac{n}{3}\rceil+1$.

    Now let $k\ge 3$ be a positive integer. Then, as before, by \Cref{lemcolonedge}, $(I(C_n)^{[k]}:x_1x_2)=I(C_{n-2})^{[k-1]}$. Hence, using \Cref{lemdepthextravar} and the induction hypothesis, we obtain $\d(R/(I(C_n)^{[k]}:x_1x_2))=2+\lceil\frac{n-2}{3}\rceil+k-2\ge \lceil\frac{n}{3}\rceil+k-1$. Moreover, by \Cref{cycle depth path comma lower bound}, $\d(R/(I(C_n)^{[k]}+\l x_1x_2\r))\ge \lceil\frac{n}{3}\rceil+k-1$. Consequently, $\d(R/I(C_n)^{[k]})\ge \lceil\frac{n}{3}\rceil+k-1$ (by \Cref{lemdepthequal}).
\end{proof}

Using \Cref{cycle depth lower bound} one can derive an explicit formula for $\d(R/I(C_n)^{[2]})$ as stated below. 

\begin{corollary}\label{cycle depth 2}
    Let $C_n$ denote the cycle on $n$ vertices with $n\geq 4$. Then $\d(R/I(C_n)^{[2]})= \lceil\frac{n}{3}\rceil+1$.
\end{corollary}
\begin{proof}
    By \Cref{cycle depth lower bound}, we have $\d(R/I(C_n)^{[2]})\ge \lceil\frac{n}{3}\rceil+1$. Also, it follows from \Cref{lemdepthequal} that $\d(R/I(C_n)^{[2]})\le\d(R/I(C_n)^{[2]}:x_1x_2)=\d(R/I(C_{n-2}))+2$. Thus, by \Cref{cycle edge ideal depth}, $\d(R/I(C_n)^{[2]})\le \lceil\frac{n-3}{3}\rceil+2=\lceil\frac{n}{3}\rceil+1$. Hence, $\d(R/I(C_n)^{[2]})= \lceil\frac{n}{3}\rceil+1$.
\end{proof}

\begin{proposition}\label{prop cycle depth m-1}
    Let $C_n$ be a cycle of even length. Then $\d(R/I(C_n)^{[k]})=2k-1$ for $k=\nu(C_n)-1=\lfloor\frac{n}{2}\rfloor-1$.
\end{proposition}
\begin{proof}
    Since $n$ is even, $\lfloor\frac{n}{2}\rfloor=\frac{n}{2}=m$ (say). Now, consider the ideal $(I(G)^{[m-1]}:x_5\cdots x_n)$. Then repeated applications of \Cref{lemcolonedge} give $(I(G)^{[m-1]}:x_5\cdots x_n)=I(C_4)$. Thus, by \Cref{cycle edge ideal depth} and \Cref{lemdepthextravar}, we get $\d(R/(I(G)^{[m-1]}:x_5\cdots x_n))=(n-4)+1=2(m-1)-1$. Hence, the result follows from \Cref{lemdepth}(iii) and \cite[Proposition 1.1]{EHHM12}.
\end{proof}

\section{Square-free powers of whiskered cycles}\label{sec whiskered cycle}

Our aim in this section is to provide sharp bounds for the regularity and depth of square-free powers of edge ideals of whiskered cycles. We also obtain exact formulas for the regularity and depth in the case of second square-free power.
\medskip

In order to investigate square-free powers of edge ideals of whiskered cycles, we need to consider a slightly general class of graphs. For $m\ge 3$ and $r\ge 1$, let $G_{2,m,r}$ be the graph (see also \Cref{fig:enter-label2}) with vertex set and edge set as follows:
\begin{align*}
    V(G_{2,m,r})&=\{x_i,y_j,z_k\mid 1\le i\le m,2\le j\le m,1\le k\le r\},\\
    E(G_{2,m,r})&=\{\{x_i,x_{i+1}\},\{x_1,x_m\}\{x_1,x_2\},\{x_i,y_i\},\{x_m,y_m\},\{x_1,z_k\}\mid 2\le i\le m-1,1\le k\le r\}.
\end{align*}

   Note that, if $r=1$, then $G_{2,m,1}$ is the whiskered cycle $W(C_m)$. Now define the graph class:
    \[
    \widetilde W(C_m):=\{G\mid G\cong G_{2,m,r}\text{ for some }r\ge 1\}.
    \]

\noindent For $G\in \widetilde{W}(C_m)$, we proceed to give bounds for $\reg(I(G)^{[k]})$. To begin with, we first consider the case $m=4$ separately.
    \begin{lemma}\label{whisker base case}
        Let $G\in \widetilde W(C_4)$ and $2\le k\le \nu(G)=4$. Then
        \[
        \reg(I(G)^{[k]})\le 2k+\left\lfloor\frac{4-k}{2}\right\rfloor.
        \]
    \end{lemma}
    \begin{proof}
        Without loss of generality let $G=G_{1,4,r}$ for some $r\ge 1$. Since $\nu(G)=4$, we have $k\le 4$. Note that, by \cite[Theorem 5.1]{BHZN}, $I(G)^{[4]}$ has a linear resolution, and thus, $\reg(I(G)^{[k]})= 2k+\left\lfloor\frac{4-k}{2}\right\rfloor$ when $k=4$. Therefore, we may assume that $k\in\{2,3\}$. Now, consider the ideal
        \[
        L_1=I(G)^{[k]}+\l x_2y_2,x_2x_3,x_1x_2\r.
        \]
        Using \Cref{colon with variable}, we have $(L_1:x_2)=\l y_2,x_3,x_1\r$ and $L_1+\l x_2\r=\l x_2\r+I(G\setminus x_2)^{[k]}$. Since $G\setminus x_2\in \overline{W}(P_3)$, using \Cref{proppathwhisk}, we obtain $\reg(L_1+\l x_2\r)\le 2k+\lfloor\frac{4-k}{2}\rfloor$. Thus, by \Cref{lemreg}, $\reg(L_1)\le 2k+\lfloor\frac{4-k}{2}\rfloor$. Next, consider the ideal $L_2=I(G)^{[k]}+\l x_2y_2,x_2x_3\r$. Note that $L_2+\l x_1x_2\r=L_1$. Also, by \Cref{lemcolonedge}, 
    \begin{align*}
    (L_2:x_1x_2)=\begin{cases}
        \l y_2,x_3,x_4y_4\r&\text{ if }k=2,\\
        \l y_2,x_3\r&\text{ if }k=3.
        \end{cases}
        \end{align*}
         Thus, using \Cref{lemreg}, we have $\reg(L_2)\le 2k+\lfloor\frac{4-k}{2}\rfloor$. Suppose $L_3=I(G)^{[k]}+\l x_2y_2\r$. Then $L_3+\l x_2x_3\r=L_2$ and $(L_3:x_2x_3)=\l y_2\r+(I(G)^{[k]}:x_2x_3)$. Using \Cref{lemcolonedge} again, it is easy to see that $(L_3:x_2x_3)=\l y_2\r+I(G_1)^{[k-1]}$ for some $G_1\in \overline{W}(P_2)$. Thus, we have $\reg(L_3:x_2x_3)\le 2k-2+\lfloor\frac{4-k}{2}\rfloor$ by \Cref{proppathwhisk}, and consequently, $\reg(L_3)\le 2k+\lfloor\frac{4-k}{2}\rfloor$. In other words, $\reg(I(G)^{[k]}+\l x_2y_2\r\le 2k+\lfloor\frac{4-k}{2}\rfloor$. Now applying \Cref{lemcolonedge} one more time, we see that $(I(G)^{[k]}:x_2y_2)=I(G_2)^{[k-1]}$, where $G_2\in \overline{W}(P_3)$, and hence, by \Cref{proppathwhisk}, $\reg((I(G)^{[k]}:x_2y_2))\le 2k-2+\lfloor\frac{4-k}{2}\rfloor$. Finally using \Cref{lemreg}, we conclude that $\reg(I(G)^{[k]})\le 2k+\left\lfloor\frac{4-k}{2}\right\rfloor$ as desired.
    \end{proof}

    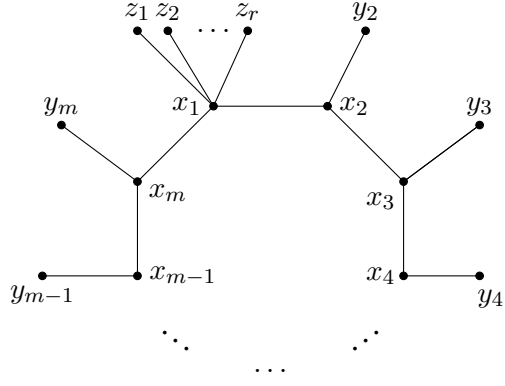
\begin{figure}
    \centering
    \begin{tikzpicture}
        [scale=.50]
        \draw [fill] (0,2.5) circle [radius=0.1]; % x_1
        \draw [fill] (-2,4.5) circle [radius=0.1]; % y_1
        \draw [fill] (-1.2,4.5) circle [radius=0.1]; % y_1
         \draw [fill] (0.9,4.5) circle [radius=0.1]; % y_1
        \draw [fill] (5,0.5) circle [radius=0.1]; % x_3
        \draw [fill] (7,2) circle [radius=0.1]; % y_3
        \draw [fill] (5,-2) circle [radius=0.1]; % x_4
        \draw [fill] (7,-2) circle [radius=0.1]; % y_4
        \draw [fill] (3,2.5) circle [radius=0.1]; % x_2
        \draw [fill] (4,4.5) circle [radius=0.1]; % w_1
        \draw [fill] (-2,0.5) circle [radius=0.1];% x_n
        \draw [fill] (-4,2) circle [radius=0.1]; %y_n
        \draw [fill] (-2,-2) circle [radius=0.1];% x_n-1
        \draw [fill] (-4.5,-2) circle [radius=0.1]; %y_n-1
        \node at (-4,2.5) {$y_m$};
        \node at (-1.2,0.2) {$x_{m}$};
        \node at (-0.7,2.5) {$x_1$};
        \node at (-2,5) {$z_1$};
        \node at (-1.2,5) {$z_2$};
        \node at (0.9,5) {$z_r$};
        \node at (4.4,0) {$x_{3}$};
        \node at (3.7,2.5) {$x_2$};
        \node at (4,5) {$y_2$};
        \node at (7,2.5) {$y_3$};
        \node at (4.4,-2) {$x_{4}$};
        \node at (7.3,-2.6) {$y_4$};
        \node at (-0.8,-2) {$x_{m-1}$};
        \node at (-4.5,-2.5) {$y_{m-1}$};
        \node at (-1,-3.5) {$\ddots$};
        \node at (4,-3.5) {\reflectbox{$\ddots$}};
        \node at (1.5,-4.5) {$\hdots$};
        \node at (0,4.5) {$\hdots$};
        \draw (0,2.5)--(3,2.5);
        \draw (0,2.5)--(-2,0.5);
        \draw (4,4.5)--(3,2.5);
        \draw (3,2.5)--(5,0.5);
        \draw (0,2.5)--(-2,4.5);
        \draw (0,2.5)--(-1.2,4.5);
        \draw (0,2.5)--(0.9,4.5);
        \draw (-4,2)--(-2,0.5);
        \draw (7,2)--(5,0.5);
        \draw (-2,-2)--(-2,0.5);
        \draw (-2,-2)--(-4.5,-2);
        \draw (7,2)--(5,0.5);
        \draw (5,0.5)--(5,-2)--(7,-2);
    \end{tikzpicture}
    \caption{A graph $G\in \widetilde W(C_m)$}
    \label{fig:enter-label2}
\end{figure}

    Now we are ready to prove the first main result of this section.

\begin{theorem}\label{reg whiskeredcycle multiedge}
    Let $m\ge 3$ be an integer and $G\in \widetilde W(C_m)$. Let $k$ be an integer such that $2\le k\le\nu(G)=m$. Then we have
    \[
    2k+\left\lfloor\frac{m-k-1}{2}\right\rfloor\le \reg(I(G)^{[k]})\le 2k+\left\lfloor\frac{m-k}{2}\right\rfloor.
    \]
\end{theorem}
\begin{proof}
    Without loss of generality, let $G=G_{2,m,r}$. First, we show that $\reg(I(G)^{[k]})\le 2k+\left\lfloor\frac{m-k}{2}\right\rfloor$ and we proceed by induction on $m$. If $m=3$, then it is easy to see that $G$ is a co-chordal graph, and thus, $\reg(I(G)^{[k]})=2k$. If $m=4$, then by \Cref{whisker base case}, $\reg(I(G)^{[k]})\le 2k+\left\lfloor\frac{4-k}{2}\right\rfloor$. Therefore, we may assume that $m\ge 5$.

    Consider the ideal $L_1=I(G)^{[k]}+\l x_2y_2,x_2x_3,x_1x_2\r$. Using \Cref{colon with variable} we have $(L_1:x_2)=\l y_2,x_3,x_1\r+I(G_1)^{[k]}$, where $G_1$ is the induced subgraph of $G$ on the vertex set $V(G)\setminus\{x_1,x_2,x_3\}$. It is easy to see that $G_1\in \overline{W}(P_{m-3})$ and $I(G_1)^{[k]}=\l 0\r$ if $k\ge m-2$. Thus, by \Cref{proppathwhisk}, $\reg(L_1:x_2) \le 2k-1+\left\lfloor\frac{m-k}{2}\right\rfloor$. Note that $G\setminus x_2\in \overline{W}(P_{m-1})$, and thus, $I(G\setminus x_2)^{[k]}=\l 0\r$ if $k=m$. Since $L+\l x_2\r=I(G\setminus x_2)^{[k]}+\l x_2\r$, using \Cref{proppathwhisk}, we obtain $\reg(L_1+\l x_2\r)\le 2k+\lfloor\frac{m-k}{2}\rfloor$. Consequently, by \Cref{lemreg},
    \[
    \reg(L_1)\le 2k+\left\lfloor\frac{m-k}{2}\right\rfloor.
    \]
     Next, consider the ideal $L_2=I(G)^{[k]}+\l x_2y_2,x_2x_3\r$. Note that $L_2+\l x_1x_2\r=L_1$. Also, $(L_2:x_1x_2)=\l y_2,x_3\r+(I(G)^{[k]}:x_1x_2)$. Using \Cref{lemcolonedge} again, we see that $(L_2:x_1x_2)=\l y_2,x_3\r+I(G_1)^{[k-1]}$. Again, $I(G_1)^{[k-1]}=\l 0\r$ if $k=m-1$ or $k=m$. Thus, using \Cref{proppathwhisk}, we obtain $\reg(L_2:x_1x_2)\le 2k-2+\lfloor\frac{m-k}{2}\rfloor$, and consequently, by \Cref{lemreg}, we have
     \[
     \reg(L_2)\le 2k+\left\lfloor\frac{m-k}{2}\right\rfloor.
     \]
      Suppose $L_3=I(G)^{[k]}+\l x_2y_2\r$. Then $L_3+\l x_2x_3\r=L_2$. Moreover, $(L_3:x_2x_3)=\l y_2\r+(I(G)^{[k]}:x_2x_3)$. Using \Cref{lemcolonedge}, it is easy to see that $(L_3:x_2x_3)=\l y_2\r+I(G_2)^{[k-1]}$ for some $G_2\in \widetilde{W}(C_{m-2})$. Note that, $I(G_2)^{[k-1]}=\l 0\r$ if $k=m$. By the induction hypothesis, we get $\reg(L_3:x_2x_3)\le 2k-2+\lfloor\frac{m-k}{2}\rfloor$. Consequently, using \Cref{lemreg}, we have
      \[
      \reg(L_3)\le 2k+\left\lfloor\frac{m-k}{2}\right\rfloor.
      \]
      In other words, $\reg(I(G)^{[k]}+\l x_2y_2\r)\le 2k+\lfloor\frac{m-k}{2}\rfloor$. Now using \Cref{lemcolonedge} again, one can see that $(I(G)^{[k]}:x_2y_2)=I(G\setminus x_2)^{[k-1]}$, and hence by \Cref{proppathwhisk}, $\reg(I(G)^{[k]}:x_2y_2)\le 2k-2+\lfloor\frac{m-k}{2}\rfloor$. Finally using \Cref{lemreg} one more time, we conclude that $\reg(I(G)^{[k]})\le 2k+\left\lfloor\frac{m-k}{2}\right\rfloor$.\par

      For the lower bound of $\reg(I(G)^{[k]})$, note that if $k=m$, then by \cite[Theorem 5.1]{BHZN}, $\reg(I(G))^{[k]}=2m$. Thus, $\reg(I(G)^{[k]})\ge 2k+\left\lfloor\frac{m-k-1}{2}\right\rfloor$ when $k=m$. Suppose $k\le m-1$. In this case, $G\setminus x_1\in \overline{W}(P_{m-1})$ is an induced subgraph of $G$, and hence, by \Cref{proppathwhisk} and \cite[Corollary 1.3]{EHHS}, it follows that $\reg(I(G)^{[k]})\ge\reg(I(G\setminus x_1)^{[k]})=2k+\left\lfloor\frac{m-k-1}{2}\right\rfloor$. This completes the proof.
\end{proof}

Note that, by \cite[Proposition 1.1]{MFY}, $\reg(I(W(C_m)))=\lfloor\frac{m}{2}\rfloor+1$. Thus as a corollary of \Cref{reg whiskeredcycle multiedge} we obtain the following.

\begin{corollary}\label{whisker reg bound}
    Let $G=W(C_m)$. Then for each $1\le k\le \nu(G)=m$, we have
    \[
    2k+\left\lfloor\frac{m-k-1}{2}\right\rfloor\le \reg(I(G)^{[k]})\le 2k+\left\lfloor\frac{m-k}{2}\right\rfloor.
    \]
\end{corollary}

\begin{remark}\label{whisker not linear}
In view of \Cref{whisker reg bound}, one can see that $I(W(C_m))^{[k]}$ does not have a linear resolution if $k<m-2$.
\end{remark}

\begin{remark}
    Let $G=W(C_m)$. Then it follows from \cite[Proposition 1.1]{MFY} that $\reg(I(G))=1+\lfloor \frac{m}{2}\rfloor=2+\lfloor\frac{m-2}{2}\rfloor$. Also, observe that when $m-k$ is odd, \Cref{whisker reg bound} gives $\reg(I(G)^{[k]})=2k+\left\lfloor\frac{m-k-1}{2}\right\rfloor$. Based on these facts and our next result on $\reg(I(G)^{[2]})$, we predict that $\reg(I(G)^{[k]})=2k+\lfloor\frac{m-k-1}{2}\rfloor$ for all $1\le k\le m$ (see \Cref{conj whisker reg}).
\end{remark}

\begin{theorem}\label{reg multi-edge whiskered cycle second power}
    Let $m\ge 3$ be an integer and $G\in \widetilde W(C_m)$. Then
    \[
    \reg(I(G)^{[2]})= 4+\left\lfloor\frac{m-3}{2}\right\rfloor.
    \]
\end{theorem}
\begin{proof}
    In view of \Cref{reg whiskeredcycle multiedge}, it is enough to prove that $\reg(I(G)^{[2]})\le 4+\left\lfloor\frac{m-3}{2}\right\rfloor$. First, consider the ideal $(I(G)^{[2]}:x_2x_3)$. By \Cref{lemcolonedge}, $(I(G)^{[2]}:x_2x_3)=I(G_1)$, where $V(G_1)=V(G)\setminus\{x_2,x_3\}$, and $E(G_1)=E(G\setminus \{x_2,x_3\})\cup\{\{x_1,x_4\},\{x_1,y_3\},\{x_4,y_2\},\{y_2,y_3\}\}$. Now $(I(G_1):x_1)=\l z_k,y_3,x_4,x_m\mid 1\le k\le r\r+I(G_2)$, where $G_2=G_1\setminus\{x_1,y_2,y_3,z_k,x_4,x_m\mid 1\le k\le r\}$. Note that $G_2\in \overline{W}(P_{m-5})$, and $I(G_2)=\l 0\r$ if $m\le 5$. By \Cref{proppathwhisk}, we have $\reg(I(G_1):x_1)\le 1+\left\lfloor\frac{m-3}{2}\right\rfloor$. Moreover, $I(G_1)+\l x_1\r=I(G_3)$ for some $G_3\in \overline{W}(P_{m-2})$, where $I(G_3)=\l 0\r$ if $m\le 2$. Thus, again by \Cref{proppathwhisk}, $\reg(I(G_1)+\l x_1\r)\le 2+\left\lfloor\frac{m-3}{2}\right\rfloor$, and consequently, by \Cref{lemreg}, $\reg(I(G_1))\le 2+\left\lfloor\frac{m-3}{2}\right\rfloor$. Our aim now is to show that $\reg(I(G)^{[2]}+\l x_2x_3\r)\le 4+\left\lfloor\frac{m-3}{2}\right\rfloor$.\par 

    Consider the ideal $L_1=I(G)^{[2]}+\l x_2x_3,x_1x_2,x_2y_2\r$. Using \Cref{lemcolonedge}, we have $(L_1:x_2)=\l y_2,x_3,x_1\r+I(G_4)^{[2]}$, where $G_4$ is the induced subgraph of $G$ on the vertex set $V(G)\setminus\{x_1,x_2,x_3\}$. It is easy to see that $G_4\in \overline{W}(P_{m-3})$, and $I(G_4)^{[2]}=\l 0\r$ if $m\le 4$. Thus, by \Cref{proppathwhisk}, $\reg(L_1:x_2) \le 3+\left\lfloor\frac{m-3}{2}\right\rfloor$. Note that $L_1+\l x_2\r=I(G\setminus x_2)^{[2]}+\l x_2\r$, where $G\setminus x_2\in \overline{W}(P_{m-1})$. Hence, using \Cref{proppathwhisk}, we obtain $\reg(L_1+\l x_2\r)\le 4+\lfloor\frac{m-3}{2}\rfloor$. Consequently, by \Cref{lemreg},
    \[
    \reg(L_1)\le 4+\left\lfloor\frac{m-3}{2}\right\rfloor.
    \]
     Next, consider the ideal $L_2=I(G)^{[2]}+\l x_2x_3,x_1x_2\r$. Note that $L_2+\l x_2y_2\r=L_1$. Also, $(L_2:x_2y_2)=\l x_1,x_3\r+(I(G)^{[2]}:x_2y_2)$. Using \Cref{lemcolonedge}, we see that $(L_2:x_2y_2)=\l x_1,x_3\r+I(G_4)$. Again, $I(G_4)=\l 0\r$ if $m=3$. Thus, by \Cref{proppathwhisk}, we obtain $\reg(L_2:x_2y_2)\le 2+\lfloor\frac{m-3}{2}\rfloor$, and consequently, by \Cref{lemreg}, we have
     \[
     \reg(L_2)\le 4+\left\lfloor\frac{m-3}{2}\right\rfloor.
     \]
      Now, let $L_3=I(G)^{[2]}+\l x_2x_3\r$. Then $L_3+\l x_1x_2\r=L_2$. Moreover, $(L_3:x_1x_2)=\l x_3\r+(I(G)^{[2]}:x_1x_2)$. Using \Cref{lemcolonedge}, it is easy to see that $(L_3:x_1x_2)=\l x_3\r+I(G_5)$ for some $G_5\in \overline{W}(P_{m-2})$. Therefore, by \Cref{proppathwhisk}, $\reg(L_3:x_1x_2)\le 2+\lfloor\frac{m-3}{2}\rfloor$. Consequently, using \Cref{lemreg}, we have
      \[
      \reg(L_3)\le 4+\left\lfloor\frac{m-3}{2}\right\rfloor.
      \]
      In other words, $\reg(I(G)^{[2]}+\l x_2x_3\r)\le 4+\lfloor\frac{m-3}{2}\rfloor$. Recall that, we already proved $\reg(I(G)^{[2]}: x_2x_3)\le 2+\left\lfloor\frac{m-3}{2}\right\rfloor$. Hence, by \Cref{lemreg}, $\reg(I(G)^{[2]})= 4+\left\lfloor\frac{m-3}{2}\right\rfloor$ as desired.
\end{proof}

As a corollary of \Cref{reg multi-edge whiskered cycle second power}, we obtain the following formula for regularity of second square-free power of edge ideals of whiskered cycles.

\begin{corollary}\label{whisker cycle reg 2}
    Let $G=W(C_m)$. Then $\reg(I(G)^{[2]})= 4+\left\lfloor\frac{m-3}{2}\right\rfloor$.
\end{corollary}

Next, let us derive the depth of second square-free powers of edge ideals of whiskered cycles. %The first result in this direction is a general upper bound for $\d(R/I(W(C_m))^{[k]})$. 

\begin{theorem}\label{whisker cycle depth 2}
    Let $W(C_m)$ denote the whiskered cycle. Then
    \begin{align*}
        \depth(R/I(W(C_m))^{[2]})=
        \begin{cases}
            3 &\text{ if $m=3$},\\
            m+1 &\text{ if $m>3$}.
        \end{cases}
    \end{align*}

\end{theorem}
\begin{proof}
First, we will prove the result for $m=3$. By \Cref{lemcolonedge}, we have $(I(W(C_3))^{[2]}:x_2x_3)=I(G')$, where $V(G')=\{x_1,y_1,y_2,y_3\}$ and $E(G')=\{\{x_1,y_1\},\{x_1,y_2\},\{x_1,y_3\},\{y_2,y_3\}\}$. Then $(I(G'):x_1)=\l y_1, y_2, y_3 \r$, and thus, $\depth(R/(I(G'):x_1))=3$ by \Cref{lemdepthextravar}. We apply \Cref{lemdepth}(iii) to get $\depth(R/I(W(C_3))^{[2]})\leq 3$. Hence, due to \cite[Proposition 1.1]{EHHM12}, it follows that $\depth(R/I(W(C_3))^{[2]})= 3$. \par 

Now, let us assume $m>3$. For simplicity of notation, we write $G=W(C_m)$. First consider the ideal $L_1=I(G)^{[2]}+\l x_2x_3,x_1x_2,x_2y_2\r$. Then using \Cref{colon with variable}, we get $(L_1:x_2)=\l y_2,x_3,x_1\r+I(G_1)^{[2]}$, where $G_1$ is the induced subgraph of $G$ on the vertex set $V(G)\setminus\{x_1,x_2,x_3\}$. One can observe that $G_1$ is the whisker graph on $P_{m-3}$. Note that if $m=4$, then $I(G_1)^{[2]}=\l 0\r$, and in this case, $\depth(R/(L_1:x_2))=5$ by \Cref{lemdepthextravar}. If $m>4$, then we have $\depth(R/(L_1:x_2))=(m-3)+2-1+3=m+1$ by \Cref{lemdepthextravar} and \Cref{thmcmtreedepth}. Therefore, in both cases, we have $\depth(R/(L_1:x_2))=m+1$. Now observe that $L_1+\l x_2\r=I(G\setminus x_2)^{[2]}$, and $G\setminus x_2$ is the whisker graph on $P_{m-1}$. Thus, using \Cref{lemdepthextravar} and \Cref{thmcmtreedepth}, we get $\depth(R/(L_1+\l x_2\r))=(m-1)+2-1+1=m+1$. Consequently, by \Cref{lemdepth},
\[
\depth(R/L_1)=m+1.
\]
Next, consider the ideal $L_2=I(G)^{[2]}+\l x_2x_3,x_1x_2\r$. Note that $L_2+\l x_2y_2\r=L_1$. Also, $(L_2:x_2y_2)=\l x_1,x_3\r+(I(G)^{[2]}:x_2y_2)$. Using \Cref{lemcolonedge}, we see that $(L_2:x_2y_2)=\l x_1,x_3\r+I(G_1)$. Again by \Cref{lemdepthextravar} and \Cref{thmcmtreedepth}, we obtain $\depth(R/(L_2:x_2y_2))=(m-3)+4=m+1$. Hence, by \Cref{lemdepth}, we have
     \[
     \depth(R/L_2)=m+1.
     \]
      Now, let us take $L_3=I(G)^{[2]}+\l x_2x_3\r$. Then $L_3+\l x_1x_2\r=L_2$. Moreover, $(L_3:x_1x_2)=\l x_3\r+(I(G)^{[2]}:x_1x_2)$. Using \Cref{lemcolonedge}, it is easy to verify that $(L_3:x_1x_2)=\l x_3\r+I(G_2)$, where $G_2$ is the whisker graph on $P_{m-2}$. Therefore, by \Cref{lemdepthextravar} and \Cref{thmcmtreedepth}, $\depth(R/(L_3:x_1x_2))=(m-2)+3=m+1$. Consequently, using \Cref{lemdepth}, we have
      \[
      \depth(R/L_3)=m+1.
      \]
      Now, consider the ideal $(I(G)^{[2]}:x_2x_3)$. By \Cref{lemcolonedge}, $(I(G)^{[2]}:x_2x_3)=I(G_3)$, where $V(G_3)=V(G)\setminus\{x_2,x_3\}$, and $E(G_3)=E(G\setminus \{x_2,x_3\})\cup\{\{x_1,x_4\},\{x_1,y_3\},\{x_4,y_2\},\{y_2,y_3\}\}$. Next, we focus on the ideal $(I(G_3):x_1)$ and it is easy to check that for $m\leq 5$, $(I(G_3):x_1)$ is generated by some variables and $\depth(R/(I(G_3):x_1))=m+1$ by \Cref{lemdepthextravar}. Now, we assume $m>5$. Then $(I(G_3):x_1)=\l y_1,y_3,x_4,x_m\r+I(G_4)$, where $G_4=G_3\setminus\{x_1,y_1, y_2,y_3,x_4,x_m\}$. Note that $G_4$ is the whisker graph on $P_{m-5}$. By \Cref{lemdepthextravar} and \Cref{thmcmtreedepth}, we have $\depth(R/(I(G_3):x_1))=(m-5)+6=m+1$. Again, $I(G_3)+\l x_1\r=I(G_3\setminus\{x_1\})+\l x_1\r$, where $G_3\setminus \{x_1\}$ is the whisker graph on $P_{m-2}$. Thus, by \Cref{lemdepthextravar} and \Cref{thmcmtreedepth}, $\depth(R/(I(G_3)+\l x_1\r))=(m-2)+3=m+1$. Consequently, we have $\depth(R/I(G_3))=\depth(R/(I(G)^{[2]}:x_2x_3))=m+1$ by \Cref{lemdepth}. Hence, using \Cref{lemdepth} again, we obtain $\depth(R/I(G)^{[2]})=m+1$.
\end{proof}

The following proposition gives an upper bound of $\d(R/I(W(C_m))^{[k]})$ for any $1\le k\le m$, and this bound is sharp in some cases.

\begin{proposition}\label{depth whisker upper bound}
Let $W(C_m)$ denote the whiskered cycle. Then for each $1\le k\le \nu(W(C_m))=m$, we have
\[
\d(R/I(W(C_m))^{[k]})\le m+k-1
\]
\end{proposition}
\begin{proof}
    If $k=1$, the assertion follows from \cite[Theorem 2.4]{CM}. For $k\ge 2$, we have from \Cref{lemdepth} that $\d(R/I(W(C_m))^{[k]})\le \d(R/(I(W(C_m))^{[k]}:x_1y_1))$ . Again, it follows from \Cref{lemcolonedge} that $(I(W(C_m))^{[k]}:x_1y_1)=I(G)^{[k-1]}$, where $G=W(P_{m-1})$. Therefore, applying \Cref{lemdepthextravar} and \Cref{thmcmtreedepth}, we obtain $\d(R/I(W(C_m))^{[k]})\le m-1+(k-1)-1+2=m+k-1$.
\end{proof}

\section{Some Conjectures}\label{sec conj}

Based on our results from previous sections and some computations through the computer algebra system Macaulay2 \cite{GS}, we predict formulas for the depth of square-free powers for cycles and whiskered cycles, and the regularity of square-free powers for whiskered cycles.\medskip 

 In view of \Cref{cycle depth lower bound}, \Cref{cycle depth 2}, \Cref{prop cycle depth m-1} and computational evidence, we propose the following conjecture.

\begin{conjecture}\label{conj depth cycle}
    Let $C_n$ denote a cycle of length $n$. Then
    \begin{align*}
        \depth(R/I(C_n)^{[k]})=\begin{cases}
             \lceil\frac{n}{3}\rceil +k-1 &\text{ if } k=2,\ldots, \lceil\frac{n}{3}\rceil,\\
             2k-1 &\text{ if } k=\lceil\frac{n}{3}\rceil+1,\ldots, \nu(C_n).
        \end{cases}
    \end{align*}
\end{conjecture}
\noindent Note that the cases $k=1, 2, \nu(C_n)$ of the above conjecture follows from \Cref{cycle edge ideal depth}, \Cref{cycle depth 2} and \Cref{depth highest power}, respectively. Moreover, for cycles of even length, we show in \Cref{prop cycle depth m-1} that the conjecture is true for $k=\nu(C_n)-1$. Also, we have verified the conjecture up to $n=15$ using Macaulay2 and considering the base field as $\mathbb{Q}$.\par

Now, let us focus on the regularity and depth of square-free powers of edge ideals of whiskered cycles. We have shown in \Cref{whisker reg bound} that for each $1\le k\le \nu(W(C_m))=m$, 
\[
2k+\left\lfloor\frac{m-k-1}{2}\right\rfloor\le \reg(I(W(C_m))^{[k]})\le 2k+\left\lfloor\frac{m-k}{2}\right\rfloor.
\]
Observe that whenever $m-k$ is odd, $\reg(I(W(C_m))^{[k]})=2k+\left\lfloor\frac{m-k-1}{2}\right\rfloor$. Moreover, $\reg(I(W(C_m)))$ is equal to the lower bound given above (see \cite[Proposition 1.1]{MFY} or \cite[Theorem 13]{BT2013}). We have shown in \Cref{reg multi-edge whiskered cycle second power} that $\reg(I(W(C_m))^{[2]})$ also coincides with the above lower bound. Based on these facts, we state the following conjecture.

\begin{conjecture}\label{conj whisker reg}
    Let $W(C_m)$ denote the whisker graph on $C_m$. Then for all $1\leq k\leq m$,
    $$\reg(I(W(C_m))^{[k]})=2k+\left\lfloor\frac{m-k-1}{2}\right\rfloor.$$
\end{conjecture}
\noindent As before, \Cref{conj whisker reg} has been verified up to $m=10$ using Macaulay2 and taking $\mathbb{K=Q}$.\par 

In \Cref{whisker cycle depth 2}, we explicitly derive $\depth(R/I(W(C_m))^{[2]})$. Also, it is well-known \cite{CM} that for any graph $G$ on $m$ vertices, $R/I(W(G))$ is Cohen-Macaulay with $\depth(R/I(W(G)))=m$. These results and our computation suggest the following conjecture.
\begin{conjecture}\label{conj whisker depth}
    Let $W(C_m)$ denote the whisker graph on $C_m$. Then
    \begin{align*}
        \depth(R/I(W(C_m))^{[k]})=\begin{cases}
             m +k-1 &\text{ if } k=1,\ldots, \lfloor\frac{m}{2}\rfloor,\\
             2k-1 &\text{ if } k=\lfloor\frac{m}{2}\rfloor+1,\ldots,m.
        \end{cases}
    \end{align*}
    
\end{conjecture}

\noindent
{\bf Acknowledgements.} The first and the second authors are supported by Postdoctoral Fellowships at Chennai Mathematical Institute. The third author would like to thank the National Board for Higher Mathematics (India) for the financial support through the NBHM Postdoctoral Fellowship. All the authors are partially supported by a grant from the Infosys Foundation.  

\subsection*{Data availability statement} Data sharing does not apply to this article as no new data were created or
analyzed in this study.

\subsection*{Conflict of interest} The authors declare that they have no known competing financial interests or personal
relationships that could have appeared to influence the work reported in this paper.

\bibliographystyle{abbrv}
\bibliography{ref}
\end{document}